\documentclass[12pt]{amsart}
\usepackage{amsmath,amssymb,amsfonts,latexsym}
\usepackage{mathrsfs}
\usepackage[dvips]{geometry}
\usepackage{array}
\usepackage{epsf}
\usepackage{epsfig}
\newtheorem*{lemma}{Lemma}
\newtheorem*{prop}{Proposition}
\newtheorem*{thm}{Theorem}
\newtheorem*{cor}{Corollary}

\newcommand{\twoheaddownarrow}{\overset{\sim}{\twoheaddownarrow}}

\newcommand{\nc}{\newcommand}
\nc{\Ker}{\operatorname{Ker}} \nc{\rker}{\operatorname{rKer}}
\nc{\im}{\operatorname{Im}}
\nc{\stab}{\operatorname {Stab}}
\nc{\ann}{\operatorname {Ann}}
\nc{\Id}{\operatorname {Id}}
\nc{\Prim}{\operatorname {Prim}}
\nc{\Real}{\operatorname {Re}}
\nc{\Ext}{\operatorname {Ext}}
\nc{\rad}{\operatorname {rad}}
\nc{\rk}{\operatorname {rank}}
\nc{\Aut}{\operatorname {Aut}}
\nc{\supp}{\operatorname {supp}}
\nc{\height}{\operatorname {ht}}

\usepackage{array}
\usepackage{epsf}
\usepackage{epsfig}
\usepackage{xcolor}

\usepackage{tikz,tikz-cd}

\usetikzlibrary{matrix}

\usetikzlibrary{fit}

\usetikzlibrary{matrix,backgrounds}
\pgfdeclarelayer{myback}
\pgfsetlayers{myback,background,main}

\tikzset{mycolor/.style = {line width=1bp,color=#1}}%
\tikzset{myfillcolor/.style = {draw,fill=#1}}%
\usetikzlibrary{arrows,matrix,positioning}
\usepackage{pstricks}

\usetikzlibrary{graphs,graphs.standard}
\usepackage{changepage}
\usepackage{tikz-cd}
\usepackage{mathrsfs}

\makeatletter
\newcommand*{\encircled}[1]{\relax\ifmmode\mathpalette\@encircled@math{#1}\else\@encircled{#1}\fi}
\newcommand*{\@encircled@math}[2]{\@encircled{$\m@th#1#2$}}
\newcommand*{\@encircled}[1]{%
  \tikz[baseline,anchor=base]{\node[draw,circle,outer sep=0pt,inner sep=.2ex] {#1};}}
\usepackage{float}

\begin{document}

\title [Composition Tableau]{The Composition Tableau and Reconstruction of the Canonical Weierstrass Section for Parabolic Adjoint Action in type $A$}
\author [Yasmine Fittouhi and Anthony Joseph]{Yasmine Fittouhi and Anthony Joseph}
\date{\today}
\maketitle
\vspace{-.9cm}\begin{center}

Department of Mathematics\\
The University of Haifa\\
Haifa, 3498838, Israel\\
fittouhiyasmine@gmail.com
\end{center}\

\

\

\vspace{-.9cm}\begin{center}
Donald Frey Professional Chair\\
Department of Mathematics\\
The Weizmann Institute of Science\\
Rehovot, 7610001, Israel\\
anthony.joseph@weizmann.ac.il
\end{center}\

\

\date{\today}
\maketitle



 \

\textbf{Abstract}.

A ``Composition map'' is constructed, leaning heavily on earlier work [Y. Fittouhi and A. Joseph, Parabolic adjoint action, Weierstrass sections and components of the nilfibre in type $A$, Indag Math. and Y. Fittouhi and A. Joseph, The canonical component of the nilfibre for parabolic adjoint action, Weierstrass sections in type $A$, preprint, Weizmann, 2021].  It defines a composition tableau which  recovers the ``canonical'' Weierstrass section $e+V$ described in the first paper above.  Moreover \textit{without reference to this earlier work}, it is then shown that $e+V$ is indeed a Weierstrass section. This results in a huge simplification.  Moreover one may read off from the composition tableau the ``VS quadruplets'' of the second of the above papers, thereby describing the ``canonical component'' of the nil-fibre in which $e$ lies but does not of itself determine.

\section {Introduction} \label {1}
 \subsection {The Notion of a Weierstrass Section}\label {1.1}

Let $P$ be a connected algebraic group acting by morphisms on a vector space $\mathfrak m$ (necessarily finite-dimensional). Let $I:=\mathbb C[\mathfrak m]^P$ be the algebra of the $P$ invariant functions on $\mathfrak m$.  A Weierstrass section $e+V$ for this action is a linear translate of a subspace $V$ of $\mathfrak m$ by an element $e \in \mathfrak m$ such that the restriction of $I$ to $e+V$ induces an isomorphism of $I$ onto $\mathbb C[V]$.  It results in a ``canonical form'' for most $P$ orbits in $\mathfrak m$. Of course for this $I$ needs to be polynomial which is rather unusual and even this is no guarantee \cite [11.4, Example 2]{J3}.

Popov and Vinberg \cite {PV} summarized early work on this subject with $P$ reductive. Notably they explain how it leads to Weierstrass canonical form for elliptic curves. Let $\mathscr N \subset \mathfrak m$ denote the zero variety of the augmentation $I_+$. A tricky point here is to show that $\mathscr N$ is irreducible \cite [2.5]{J4}, for which there are few techniques and can often fail \cite [11.2.2] {J1}, \cite [5.4.1]{FJ1}.

\subsection {Adapted Pairs}\label {1.2}

 The ``time-honoured'' way to obtain a Weierstrass section is by constructing an adapted  pair $(e,h)$ with  $e \in \mathscr N$ and $h$ a semisimple element  of $\mathfrak p:=\text{Lie} \ P$ such that $e$ is an $h$ eigenvector and $P.e$ is dense in an irreducible component of $\mathscr N$, that is to say $e$ is regular.   Then $V$ is taken to be an $h$ stable complement of $\mathfrak p.e$ in $\mathfrak m$.  Finally a favourable comparison of degrees of generators and eigenvalues of $h$ on V gives the required result.  This was the point of view taken in the classical work of Kostant \cite {K} for $P=G$, a simple Lie group. In this he appealed to the entire principal s-triple $(e,h,f)$. By a further coincidence $e$ also generated an orbit of maximal dimension in $\mathfrak g=\text{Lie} \ G$.

If $P$ is not reductive, it may be necessary to cut down $P$ slightly, for example to its derived group $P'$, so that it admits no non-trivial semi-invariants and this will be assumed without further mention.

Having an adapted pair works rather well for biparabolic co-adjoint action \cite {J1} in type $A$ and even outside type $A$ \cite {F}.   It can even lead (in a backhanded fashion - by improving bounds on $I$ \cite [6.11] {J2}) to showing that $I$ is polynomial, calculations which have been taken further notably by Fauquant-Millet \cite {F}.

A backhanded way of obtaining a Weierstrass section when no adapted pair exists is described in \cite [Sects. 8,9]{J3} for $P$ being the Borel $B$ of a simple Lie algebra $\mathfrak g$. Yet this may lead to ``undesirable'' properties of $B(e+V)$ \cite [11.4, Example 3]{J3}, which can also be expected if $e$ is not regular.

\subsection {Adjoint Action}\label {1.3}

If $P$ is not reductive, the invariants for co-adjoint and adjoint action can behave rather differently - see \cite [3.2]{FJ1} for example.  Correspondingly one might expect differences in constructing Weierstrass sections.

For  the adjoint action of a parabolic $P$ on the Lie algebra of its nilradical $\mathfrak m$ which is the case considered here, the method of adapted pairs fails miserably.  For one thing  an irreducible component of $\mathscr N$ can fail to admit a dense $P$ orbit \cite [6.10.7]{FJ2}. (We expect that if each column is lengthened by one, then  $\mathscr N$ becomes irreducible.) Thus $e$ need not be regular and one may not be able to find a required semisimple element $h$ - \cite [3.2]{FJ2}.

In the case of the adjoint action of a parabolic $P$ acting on its nilradical $\mathfrak m$, in type $A$ we gave a construction of a Weierstrass section $e+V$ by a combinatorial construction  \cite [5.4]{FJ2}.  It was extremely complicated involving all sorts of exotic concepts: viz. Gating \cite [4.2]{FJ2}, Up-going linkage, Adjacency \cite [5.4.3]{FJ3}, Gathering of loose ends, Minimal distance criterion \cite [5.4.7]{FJ3} and so on.  In this $e$ was often far from regular.

In the first three fifths of this paper we show how our previous work, specifically \cite [3.6]{FJ3}, gives rise to a ``composition tableau''.  Then we show (Proposition \ref {4.9}) how this tableau can be \textit{directly constructed} by a method which is so easy (see \ref {4.7}) that it could have been carried out by the elderly ladies who in the distant past inspected and collected commuters' tickets in the London Underground (or indeed in the Paris metro).  Then we show in the last two fifths of this paper that \textit{independent of previous work} the resulting $e+V$ is indeed a Weierstrass section.  Yet we see no way that this could have been guessed (or even suspected) a priori.  Moreover to our delight and astonishment many of the previous exotic concepts needed for our construction \cite [5.4]{FJ2} of a Weierstrass section, drop out automatically as if by magic.  In brief our previous method had a certain internal logic, the present work has none (that we can see) except that it works.

We should emphasize that the present work by no means obviates the whole of \cite {FJ1}, \cite {FJ2} and \cite {FJ3} but mainly just \cite [4.4]{FJ2},\cite [5.4]{FJ2} and \cite [3.4,3.5]{FJ3}.  That is one good reason for needing to prove that the present work gives the same Weierstrass section even if the proof is rather wearisome.

A fundamental question is what our present construction means and how it can be wider applied (for example for the other components of the nilfibre and outside type $A$).  At present we have no answer and it remains a very significant challenge.

\subsection {Summary of Contents}\label {1.4}
Let $n$ be a positive integer and $\textbf{M}$ the set of $n\times n$ matrices. Those of determinant $1$ form a multiplicative group isomorphic to $SL(n)$.

A parabolic subgroup (of $SL(n)$) will always be one containing the Borel subgroup $B$ of upper triangular matrices of $\textbf{M}$.

Recall \cite [1.3]{FJ3}  that for the action of a parabolic $P$ on the nilradical of its Lie algebra $\mathfrak m$, there is a ``canonical'' Weierstrass section $e+V$.
Amazingly it is determined by a family of ``composition  maps'' \cite  [Lemma 3.6.8]{FJ3} all of which have the same structure.

We show (Theorem \ref {3.6}) that this structure is given rather simply in terms of the notion of a column staircase.  This leads to a composition tableau (Sect. \ref {4}) which determines $e+V$, but makes no particular reference to the column staircases and can be computed in a simple and independent manner (\ref {4.7}-\ref{4.9}) and that indeed $e+V$ is a Weierstrass section (Theorem \ref {5.8} and Section \ref {5.9}).

Let $\mathscr N^e$ be the irreducible component of $\mathscr N$ described in \cite [Cor. 6.9.8]{FJ2} as a $B$ saturation set of an appropriate subspace of $\mathfrak m$. We remark that the vector space $E_{VS} \subset \mathscr N^e$, described in \cite [6.10.4]{FJ2}, is determined by the pair $(e,V)$ and can be simply read off from the Composition Tableau - see \ref {5.10}.  Crucially \cite [Prop. 4.5.3]{FJ3} it has the property that $\overline{PE_{VS}}=\mathscr N^e$, whilst the inclusion $\overline {Pe_{VS}}\subset \mathscr N^e$ is generally strict.  Worse still $\mathscr N^e$ need not have a dense $P$ orbit \cite [6.10.7]{FJ2}. We call $\mathscr N^e$ the ``canonical component'' of $\mathscr N$.  We expect that it will nearly always be the only component, for example if there are no degree $1$ invariants.


\section {Preliminaries}\label{2}

\subsection {Tableaux} \label {2.1}

Given $u\leq v$ integers, set $[u,v]:= \{i \in \mathbb Z|u\leq i \leq v\}$.

Let  $\textbf{c}:=(c_1,c_2,\ldots,c_k)$ be a composition of $n$, that is to say a set $\textbf{c}$ of positive integers which sum to $n$.

Let $\mathscr D$ be the diagram given by a set of columns $\{C_i\}_{i=1}^k$ numbered from left to right with $C_i$ of $\height C_i$ being $c_i$. In this each column is viewed as a vertical array of descending square boxes all starting from the same row $R_1$.

Set $[\mathscr D]=\cup_{i=1}^k c_i$ and $|\mathscr D|=\max_{i=1}^k c_i$.

  Let
$\mathscr T$, be the tableau obtained from $\mathscr D$ in which the integers $1,2,\ldots,n$ are inserted sequentially in the boxes, first down the columns and then on going from left to right.

As in \cite [2.3]{FJ2}, let $\{R_j\}_{j\in \mathbb N^+}$ be the rows of $\mathscr D$ labelled with positive integers increasing from top to bottom.  For all $i \in \mathbb N^+$, set $R^i:=\{R_j\}_{j=1}^i$.

A box is labelled as $b_{i,j}$ if it occurs in the $i^{th}$ row and $j^{th}$ column of $\mathscr D$.  Let $x_{i,j}$ denote the corresponding co-ordinate vector in \textbf{M}.

We shall also label a box $b$ of $\mathscr D$ as $b(t)$ if it carries the label $t$ in $\mathscr T$.

Two boxes $b,b'$ in $\mathscr D$ are said to be \textbf{adjacent} in $\mathscr D$ if they can be joined by a horizontal line passing through no other boxes in $\mathscr D$.

\

\textbf {Definition}. A right going line from a box $b \in R_u$ to a box $b' \in R_v$ is said to be up-going (resp. down-going) if $u \geq v$ (resp. $u\leq v$).  Viewed as a left going line it is said to be down-going (resp. up-going).

\

This convention is natural as it is the sense in which the diagrams are drawn and matrices are labelled.  On the other hand when one constructs $\ell(\mathscr D)$ by an induction of rows $R_i:i=1,2,\ldots,$ as in \cite [5.4]{FJ2}, it is more natural to regard $R_i$ as being higher up than $R_j$ if $i<j$, just as in the designation of the floors of a building. Thus the opposite convention was implicit in \cite [5.4.8]{FJ2} and \cite [3.5]{FJ3}.  We always refer to the presence of this change when one must pay attention.   This conventional dichotomy was already present in Column shifting \cite [2.4,\textbf{N.B.}]{FJ2} where we noted that it would have been better to have adopted the French convention of drawing tableaux,  but who wants to do that!

 Besides in reading the latter part of this paper (Sect. \ref{5}) one should wash the brain clean of thoughts of previous work.

\subsection {Weierstrass Section Construction} \label {2.2}

The construction of \cite [Sect. 5]{FJ2} gives a family $\ell(\mathscr D)$ of lines joining boxes in distinct columns labelled either by a $1$ or by a $\ast$.  If in $\mathscr T$ the left (resp. right) hand box carries the label $i$ (resp. $j$), then the corresponding line, denoted as $\ell_{i,j}$, is assigned $x_{i,j}$, which we remark always belongs to $\mathfrak m$.  Then by definition, $e$ is just the sum of the co-ordinate vectors defined by the lines labelled by $1$ and $V$ is just the direct sum of the co-ordinate subspaces defined by the lines labelled by $\ast$.

We aim to construct $\ell(\mathscr D)$ much more simply.  The proof that this simpler formulation gives the same result is quite complicated relying heavily on the results of \cite [Sect. 5]{FJ2}.  However if one just wants to see why it gives $e+V$ to be a Weierstrass section then one can turn directly to the rather easy Section \ref {5}.

Recall that the Levi factor $L$ of $P$ is given by a set of square blocks $\textbf{B}_i$ along the diagonal of \textbf{M} of size $c_i$.  Let $\textbf{C}_i$ denote the rectangular block in \textbf{M}, called the $i^{th}$ column block, lying above $\textbf{B}_i$.  One has $\mathfrak m=\oplus_{i=2}^k \textbf{C}_i$.

\subsection {Induction on Number of Columns} \label {2.3}


Let $\mathscr D^k$ be the diagram obtained from $\mathscr D$ by deleting its rightmost column $C_k$.

The construction of $\ell(\mathscr D^k)$ is of course also given by \cite [Sect. 5]{FJ2}.  Moreover by \cite [Lemma 5.4.9]{FJ2}, $\ell(\mathscr D)$ is just the union of $\ell(\mathscr D^k)$ and the set $\ell(\mathscr D^k,C_k)$ of lines (with their labels) joining $\mathscr D^k$ to $C_k$.  By construction our new method will also have this property (Lemma \ref {5.2}). Thus proceeding inductively, we need only compute $\ell(\mathscr D^k,C_k)$.

\subsection {The Composition Map} \label {2.4}

When $c_k> |\mathscr D^k|$, the situation simplifies further.   First by \cite [3.6.2]{FJ3} the lines joining $C_k$ to $\mathscr D^k$ are \textit{all} labelled by a $1$.
Secondly they are independent of $c_k$, \cite [3.6.4]{FJ3}.  Thirdly there is a unique right going line $\ell_t$ labelled by $1$ to the $t^{th}$ row of $C_k$, or no line at all.  In the former case let $r_t$ be the entry of the box in $S$ from which $\ell_t$ starts, in the latter case set $r_t=0$.

Note that $r_t$ is \textit{not} the row of $\mathscr D$ from which $\ell$ starts, but the row of $\textbf{C}_k$ in which $1$ occurs in the $t^{th}$ column of $\textbf{M}$.  Thus to avoid confusion with the labelling of rows of $C_k$ it was written in boldface in \cite [3.3.2]{FJ3}, that is as $\textbf{r}_t$, in \cite [3.3.2]{FJ3}.

By \cite [5.4.8(vi)]{FJ2} a box may carry at most one right going line carrying a $1$ and so in addition the $r_t\neq 0$ are pairwise distinct.

The resulting map $\mathscr D^k \rightarrow \{r_t\}_{t=1}^{|\mathscr D^k|+1}$ of $\mathscr D^k$ to the set pairwise distinct non-negative integers, is by definition \textbf{the composition map}. On its own it does \textit{not} allow one to reconstruct the set of all left going lines from $C_k$ to $S$ with their labels \cite [3.6.7]{FJ3}.  However quite remarkably, when combined with a knowledge of $[\mathscr D^k]$, this reconstruction can be made and by simple rules \cite [Lemmas 3.4.2(i), 3.6.8]{FJ3}.

Thus to determine the canonical Weierstrass section $e+V$, it suffices to determine the composition maps relative to successive removal of columns of $\mathscr D$ from the right.

\section {Towards the Composition Map}\label{3}

\subsection {Neighbouring Columns} \label {3.1}

Recall \cite [4.1.2]{FJ1} that a pair of columns $C,C'$ of height $s$ are said to be neighbouring (of height $s$)  if there are no columns  of height $s$ strictly between them.  To each such pair one may assign a Benlolo-Sanderson invariant \cite {BS} and \cite [3.6.3]{FJ1}.  The latter form a generating set for $S(\mathfrak m)^{P'}$, which is a polynomial algebra \cite [Cor. 5.2.6]{FJ1}.

\subsection {Column Staircases} \label {3.2}

Fix $c$ be a positive integer and $u \in [1,c]$.

  \textbf{Definition}. A column staircase $\mathscr S_u^c$ in $\mathscr D$ is a maximal subset $C'_{i_j}:j\in [u,c]$ of columns in $\mathscr D$, labelled as in \ref {2.1}, with the property that $\height C'_{i_j}=j$, such that every $C'_{i_j}:j\in [u,c]$, has a left neighbour $C_{i_j}$ in $\mathscr D$ and
  such that there are no columns of height $\geq j$ strictly between $C'_{i_j},C'_{i_{j+1}}:u \leq j<c$ (which implies that $C_{i_j}$ lies to the left of the columns in $\mathscr S_u^c$).

  \textbf{Attention.}  Apart from being a left neighbour to $C_{i_j}$ and lying to the left of $\mathscr S_u^c$, there are \textit{no} other conditions on the positioning of the $C_{i_j}$.  The origin of this flexibility is the fact that only the position of the $\ast$ on horizontal lines in the Second Step \cite [4.2.1]{FJ1} is important and that we are using a ``rightmost labelling'' \cite [4.2.2]{FJ1}.  This flexibility is also apparent when we impose \ref {5.4} as part of our new construction. Moreover when we say that a column staircase $C'_{i_j}:j\in [u,c]$ lies to the right of a column $C$ we do not imply any positioning of the set of the left neighbours $C_{i_j}$ relative to $C$.  For example if $\height C=3$, all the compositions $(1,2,3,1,2),(2,1,3,1,2)$ as well as the further four in which the first two entries are pushed across $C$, give a staircase $\mathscr S^2_1$ of height $1$ and depth $2$ to the right of $C$. In this case $\mathscr S^2_1$ is deemed to be the leftmost staircase of depth $2$ to the right of $C$ - see after \textbf{Notation} below. Adjoining a column of height $2$ to the right would give a further staircase $\mathscr S^2_2$ of depth $2$ to the right of $C$ but it would not be the leftmost one of depth $2$.  Concern for these complications are obviated in our new construction.

  \

  \textbf{Notation.}
  Let $b(\mathscr S^c_u)$ be the box left adjacent to $C'_{i_u}\cap R_u$ (the latter is the lowest box in $C'_{i_u}$).  It lies in $R_u$ and in some column, noted $C'_{i_{u-1}}$, and to the right of $C_{i_u}$, possibly strictly.

  \

  When the height $u$ of the first column in $\mathscr S^c_u$ is unspecified we shall simply write $\mathscr S^c$ and if $c$ is also unspecified we shall simple write $\mathscr S$.  We call $u$ (resp. $c$) the height (resp. depth) of $\mathscr S_u^c$ and $c-u+1$ its length, in accordance with the convention of \ref {2.1}.

 A column staircase $\mathscr S_u^c$ is said to be (right) extremal if there are no columns of height $\geq c$ to the right of $C'_{i_c}$, which we recall has height $c$.


We say that a box $b$ has no right adjacent box, if a horizontal line from $b$ does not meet a box strictly to the right of $b$.


Call a box $b$ in $\mathscr D$ right extremal in $\mathscr D$, relative to $\ell(\mathscr D)$ if there is no right going line from $b$, labelled by a $1$.

Obviously a box $b$ in $C_k$ is right extremal, but there may be others.
On the other hand a box may have no right adjacent box, yet not be right extremal and conversely a box may be right extremal and yet have a right adjacent box.

\

 \textbf{Example.}  Consider the composition $(2,2,1,1)$.  The construction of \cite [Sect. 5]{FJ2} gives lines $\ell_{1,3},\ell_{3,5},\ell_{4,6}$ labelled by a $1$ and lines $\ell_{2,3},\ell_{5,6}$ labelled by a $\ast$. Thus $b(2)$ is right extremal with right adjacent box $b(4)$ which is in turn not right extremal.

\subsection {Extremal boxes} \label {3.3}

\begin {lemma}  Assume $c_k>|\mathscr D^k|$.  Then the boxes in $\mathscr D$ which are joined to $C_k$ by a line (necessarily labelled by a $1$) are exactly those which are right extremal in $\mathscr D^k$.
\end {lemma}

\begin {proof}
Suppose $b \in \mathscr D^k$ is joined by a line labelled by $1$ to a box in $C_k$.  Then by \cite [5.4.8(vi)]{FJ2} $b$ cannot be joined by a right going line labelled by a $1$ to a box in $\mathscr D^k$.

By \cite [5.4.9]{FJ2}, this means that $b$ is right extremal in $\mathscr D^k$.

For the converse, set $|\mathscr D^k|=s$ and suppose $c_k=s+1$.  Following \cite [6.1.1]{FJ2} we denote the last column of $\mathscr D$ by $C_k(s+1)$ to indicate its height, namely $s+1$.

 Adjoin a column $C_0$ of height $s+1$ to the left of $\mathscr D$ to obtain an ``augmented'' diagram $\hat{\mathscr D}$.  By \cite [Lemma 5.4.10]{FJ2}, this does not change the right going lines in $\mathscr D$ nor their labels, except that if there is a right going line $\ell$ to the box in $C_k(s+1)\cap R_{s+1}$ labelled (necessarily) by a $1$ in $\mathscr D$, then in $\hat{\mathscr D}$, this line is labelled by a $\ast$, or if not, a line in $R_{s+1}$ joining $C_0,C_k$ labelled by a $\ast$, is added.

  Through \cite [5.4.8(v)]{FJ2}, every box in $\hat{\mathscr D}\setminus C_k(s+1)$ has a right going line $\ell$, labelled by a $1$ with the possible exception that $\ell$ ends at the lowest box in $C_k(s+1)$, in which  case $l$ labelled by a $\ast$.

 Combining the observations of the last two paragraphs we conclude that every box $b \in \mathscr D^k$ has a right going line $\ell$ labelled by a $1$. Moreover $\ell$ is unique by \cite [5.4.8(vi)]{FJ2}.  If $b$ is right extremal in $\mathscr D^k$, then by definition $\ell$ does not meet $\mathscr D^k$ and so is right extremal.  Hence $\ell$ must meet $C_k$.  This gives the converse.
\end {proof}

\textbf{Remark.}  The proof of the converse means that $b$ cannot lie in $\mathscr D^k$ and also be right extremal in $\mathscr D$.  This is by virtue of the hypothesis, $c_k>|\mathscr D^k|$.

 \subsection {A Simple Case} \label {3.4}

  Suppose all columns in $\mathscr D$ have different heights. Then the construction of \cite [5.4]{FJ2} says that $\ell(\mathscr D)$ consists of all the horizontal lines joining adjacent boxes and are labelled by a $1$.

   Now for all $s \in [1,k]$, set $d_s:=\max_{s'>s}c_{s'}$.  Then the boxes in $C_s$ lying in rows $i\in [d_s+1,c_s]$ are exactly those which are  right extremal. Moreover there are $\height \mathscr D$ of them.

   This gives the following result.

   \begin {lemma}  Assume that all the columns in $\mathscr D^k$ have distinct heights.  Then $r_t$ is just the entry of the unique box of $\mathscr D^k$ with no right adjacent box in row $t$.
   \end {lemma}

\subsection {Replacement} \label {3.5}


The general case is  a little more complicated.  It involves the replacement of those boxes with no right adjacent box.

\textit{Assume in Section \ref {3.5} that $c_k> |\mathscr D^k|$}.

\subsubsection {The Lower Part of a Column} \label {3.5.1}

Take $i\in [1,k-1]$ and set $c^i=\max_{j|k>j>i}c_j$.  If $c_i>c^i$, we say that $C_i$ is right extremal in $\mathscr D^k$.

We call $C_i^L:=C_i\setminus (C_i\cap R^{c^i})$, the $c^i$-lower part of $C_i$.

Fix $i \in [1,k-1]$, so that $C_i$ is right extremal in $\mathscr D^k$.

Set $b_0 =b_{c^i+1,i}$.

\subsubsection {Extremality}\label {3.5.2}

\begin {lemma}  Retain the above notation. A necessary and sufficient criterion for a box $b \in C_i^L\cap R_j:j>c^i$ to be right extremal in $\mathscr D^k$, is that there is no column $C'$ in $\mathscr D^k$ of height $j-1$ to the right of $C_i$ with a left neighbour $C$. When this criterion holds, $b$ is joined by a horizontal line $\ell$ to a box in $C_k\cap R_j$ with label $1$.
\end {lemma}

\begin {proof} Recall (\ref {2.1}) that we have reversed our convention on up/down going to that used in \cite {FJ2}.

By \cite [5.4.8(iii),(ix)]{FJ2} a right going line $\ell$ can only be up-going by one row and if so is labelled by a $1$.   Thus the criterion is not satisfied except if $j=c_i=c^i+1$.

Moreover $\ell$ obtains from the joining of the pair $(b''_j,b_{u'_j})$ in \cite [5.4.7(i)]{FJ2}, and in this $b_{u'_j}$ lies in the lowest row of a column $C'_c$ of height $c:=c_i-1$ and has a left neighbour $C_c$. Then $C'_c$ is the columns of greatest height of a column staircase $\mathscr S^c$ of depth $c$. Moreover by the minimal distance criterion of \cite [5.4.8]{FJ2}, $\ell$ exists exactly when $\mathscr S^c$ is the leftmost column staircase of depth $c$ to the right of $C_i$ - (recall Definition \ref {3.2}).

When such a pair does not exist, then $b$ is extremal and so by Lemma \ref {3.3} there is line $\ell$ to $C_k$, which cannot be up-going by \cite [5.4.8(ix)]{FJ2}, because the only left going line from $C_k$ with label $\ast$ starts at $C_k\cap R_{c_k}$ and by hypothesis $c_k>\height C_i$.

Finally $C_k$ is  the only column of height $> c$ to the right of $C_i$  and this has no left neighbour since $c_k>|\mathscr D^k|$. Thus there can be no right going line from $b$ labelled by a $\ast$. Consequently there can be no right and strictly down-going line from $b$ obtained by the joining of the pair $(b_{u_j},b''_j)$ in \cite [5.4.7(i)]{FJ2}.  Thus $l$ is horizontal.
\end {proof}

\subsubsection {Extremality}\label {3.5.3}

Retain the above hypotheses and notation.  Set $b_0:=b_{c^i+1,i}$. By the above lemma, it is the only box in $C_i^L$ which might not be right extremal and this only when $c:=c^i=c_i-1$.

\begin {prop} Assume that $c_k> |\mathscr D^k|$. If  $b_0$ is not right extremal in $\mathscr D^k$, then it is joined by a right going line labelled by $1$ to a box $b_1'\in R_{c^i}$.

Moreover $b_1'$ is joined by a line labelled by a $\ast$ to a box $b_1$ which is some $b(\mathscr S^{c^i})$.

If $b_1$ is not right extremal, it is joined by a right going line labelled by $1$ to a box $b_2'\in R_{c^i}$, which in turn is joined by a left going line labelled by a $\ast$ to a box $b_2$ which is some new $b(\mathscr S^{c^i})$.

This  process continues until $\mathscr S^{c^i}$ is a right extremal column staircase in $\mathscr D^k$ and (equivalently) $b_m= b(\mathscr S^{c^i})$ is right extremal in $\mathscr D^k$.

Finally in $\mathscr D$, $b_m$ is joined to $b_{c^i+1,k}$ by a line labelled by a $1$, which is hence right and strictly down-going.

\end {prop}

\begin {proof}

The first assertion follows from \cite [5.4.8(ix)]{FJ2}, as explained in Lemma \ref {3.5.2}.  Here $b'_1=b_{u'_j}$ and lies at the bottom of a column $C'_c$ with left neighbour $C_c$ and is joined by a left going line labelled a $\ast$  to a box $b_1\in R_u:u \leq c$ (which is $b_{u_j}$ in the notation of \cite [5.4.6]{FJ2}).

  In addition, by adjacency \cite [5.4.3]{FJ2}, there is no column of height $\geq c$ (strictly) between $b_1,b_1'$. (However it can happen that $C_{c}$ lies strictly to the left of $C_i$.)

  Now in the language of \cite [6.6.2]{FJ2} suppose $C_{c},C'_{c}$ surround a second pair of neighbouring columns $C_j,C_j'$.  This means these columns have height $j=c-1$ and there are no columns of height $\geq c-1$ strictly between $C'_{c-1},C'_{c}$. Moreover by \cite [6.6.2, sixth paragraph]{FJ2} this process continues until for some $u\leq c$, the $C'_j:j=u,u+1,\ldots,c$, form a column staircase $\mathscr S^c_u$, with $b_1=b(\mathscr S^c_u)$.

   In particular $b_1$ must lie to the right of $C_i$ (though not necessarily strictly).  Then if $b_1$ is itself \textit{not} right extremal it must in turn be joined to some $b_2'$ in a column strictly to the right of $C_i$ by a line labelled by a $1$, through the joining of loose ends \cite [5.4.7(ii)]{FJ2}.  Moreover by up-going linkage \cite [5.4.3]{FJ2} one obtains $b_2' \in R_{c}$.

 This process continues with the right hand member of the pair $b_j,b_j':j=1,2,\ldots,m$ lying in $R_{c}$, until the left hand member $b_m$ lying in $R^{c}$ is right extremal in $\mathscr D^k$ and  must therefore be joined to a box in the last column $C_k$ of $\mathscr D$, by a line $\ell$ labelled by a $1$.

For the last part of the proposition, suppose first that $\mathscr D^k$ admits a column $C_0$ of height $c+1$.  Take the one nearest to the right.  By the choice of $c$, it must lie to the left of $C_i$ possibly $C_i$ itself. For the moment take $C_k$ to also have height $c+1$.  Then it is a right neighbour to $C_0$ and the resulting pair surrounds $\mathscr S^c$ in the sense of \cite [6.6.2]{FJ2}.  This extends in $\mathscr D$ the extremal column staircase $\mathscr S^c$ by one column and so there is a line joining $b(\mathscr S^c)$ and $b_{c+1,k}$ labelled by a $\ast$.  When the height of $C_k$ is increased by one, this line remains \cite [3.4.2(ii)]{FJ3}, but the $\ast$ is replaced by a $1$. When $C_k$ is further increased by one, this line and its label remain \cite [3.4.2(i)]{FJ3}.

If $\mathscr D^k$ does not admits a column $C_0$ of height $c+1$, adjoin one on its left.  Then a similar argument applies to the augmented diagram.  Moreover by \cite [5.4.10]{FJ2}, $C_0$ may be removed (because $\height C_k>c+1$) without changing the lines joining $\mathscr D^k$ to $C_k$.  This proves the last part.
 \end {proof}

\subsubsection {Definition}\label {3.5.4}

\

   Notice that in the proof of Preposition \ref {3.5.3} that $b_i$ has been replaced  successively by $b_{i+1}$, for all $i=0,1.\ldots,m-1$.

   \textbf{Definition}. We call $b_m$ of the proposition the replacement of $b_0$.

 \

  Recall that $c_k> |\mathscr D^k|$.  We have shown that a line from an extremal box in $\mathscr D^k$ to $C_k$ is down-going.  Yet if  $c_k< |\mathscr D^k|$, we shall see that this may fail (Lemma \ref {5.3.2}) but only slightly.

 \subsection {Description of the Right Extremal Boxes} \label {3.6}

 Retain the above notation.  Combining Lemmas \ref {3.3}, \ref {3.4} and Proposition \ref {3.5} we obtain

 \begin {thm}

 \

 $(i)$. Take $c \leq |\mathscr D^k|$.  Suppose $\mathscr S^c$ is a right extremal column staircase. Then $b(\mathscr S^c)$ is right extremal in $\mathscr D^k$.

  \

  $(ii)$.  The remaining right extremal boxes in $\mathscr D^k$ are those belonging to the $c$-lower part $C_i^L:=(R\setminus R^c) \cap C_i$ of each right extremal column $C_i$ except possibly that in the top row $R_{c+1}$ of $C_i^L$ and this exactly when there is a right extremal column staircase $\mathscr S^c$ (necessarily to the right of $C_i$).

 \end {thm}

 \textbf{Remarks.}  Retain the above notation.  Paradoxically a box in $C_i\setminus C_i^U$ \textit{can} be right extremal in $\mathscr D^k$, because it can be some $b(\mathscr S)$ with $\mathscr S$ being a right extremal column staircase.  For example, let $\mathscr D$ be given by the composition $(1,2,1,3)$.  Then $b(2)$ is right extremal in $\mathscr D^4$.

 If $c=|\mathscr  D^k|$, then the $c$-lower part of every column of $\mathscr D^k$ is empty.  Nevertheless there can be a right extremal column staircase $\mathscr S^c$ and so $r_{c+1}$ can be non-zero.  For example this is true when $\mathscr D$ is given by the composition $(1,1,2)$.


  \subsection {Examples} \label {3.7}

 Recall the notation of  \ref {2.1}.

  Suppose $\mathscr D^k$ is given by the composition $(2,1,1)$. Then $b(2)$ is not right extremal but it is replaced by $b(3)$ which is right extremal.

  A more exotic example occurs when $\mathscr D^k$ is the composition $(3,1,2,1,2)$.  Here $b(3)$ is not right extremal but is replaced by $b(5)$ which is right extremal.  Notice that the columns $C_4,C_5$ form a column staircase which is right extremal  with $b(5)$ its right extremal box.

  When we adjoin two further columns of heights $1,2$, then $C_4,C_5$ form a column staircase but which is no longer right extremal.  Rather $C_6,C_7$ form a column staircase which is right extremal  with $b(8)$ its right extremal box.

  Such a sequence of column staircases can be repeated indefinitely.

  \section {The Composition Tableau} \label {4}

   \subsection {Describing the Composition map} \label{4.1}

   Assume that $c_k>|\mathscr D^k|$.  For all $t \in [\mathscr D^k]$, let $C^t$ be the unique rightmost column of $\mathscr D^k$ of height $\geq t$. We may summarize the conclusions of \ref {3.4}-\ref {3.6} as follows.

   \begin {cor}

   The image $r_t$ of $t$ under the composition map is the entry $b(\mathscr S)$, for some right extremal column staircase $\mathscr S$, or of a box in $R_t\cap C^t$, if this box is right extremal.

    If $t\geq |\mathscr D^k|+1$, then $r_t=0$, except when equality holds. Then either $r_t=0$ or is again the entry of some $b(\mathscr S)$ with $\mathscr S$, a right extremal column staircase.

%

    \end {cor}

    \subsection {The Last Column} \label{4.2}

    The composition tableau $\mathscr T'$ is defined as follows.

    Let $\mathscr D'$ be the diagram with $k$ columns and arbitrarily many rows going downwards starting from $R_1$ (but see also below).

    Insert $r_t$ into the $t^{th}$ row and the rightmost column $C_k$ of $\mathscr D'$.  Now suppress the rightmost column of $\mathscr T$ and repeat the process to define the $(k-1)^{th}$ column of $\mathscr T'$.  This eventually gives to $\mathscr T'$.  We may regard a zero entry of $\mathscr T'$ to be no entry at all and then a postiori take $\mathscr D'$ to be the shape of $\mathscr T'$.  Through \ref {2.4} we may recover the canonical Weierstrass section $e+V$ from $\mathscr T'$.  The latter has the same entries as $\mathscr T$ in $\mathscr D$ and the same labelling of columns, but some entries are repeated in $\mathscr D'$.  In other words $\mathscr T$ is a sub-tableau of $\mathscr T'$.





  \subsection {Preservation of Extremality} \label{4.3}

  The following is an immediate consequence of \cite [Lemma 5.4.9]{FJ2}.

  \begin {lemma}  If $b\in \mathscr D_k$ is a right extremal in $\mathscr D$, then it is a right extremal in $\mathscr D_k$.
  \end {lemma}

  \textbf{Remark.}  The converse fails, for example if $c_k>|\mathscr D^k|+1$.


  \subsection {Reconstruction} \label{4.4}

   \begin {cor}  If a box with label $m$ occurs in column $C_u$ of $\mathscr T$ and in column $C_v$ of $\mathscr T'$, then it must also occur in columns $C_i:i \in [u,v]$ of $\mathscr T'$.
   \end {cor}

  \subsection {The Profile of a Column Staircase} \label{4.5}

  \

  Recall \ref {3.2}.

  \textbf{Definition}.  The profile of a column staircase $\mathscr S_u^c$ is the set of integers $i'_j:j \in [u-1,c]$.

The profile staircase of $\mathscr S_u^c$ is
$C_{j'}\cap R_{m+1}: j'\in [i'_m,i'_{m+1}-1]:m \in [u-1,c]$.

  \subsection {Recovery of Profiles in $\mathscr T'$} \label{4.6}
  \subsubsection {An Extremal box in an Extremal Column}\label {4.6.1}

  \begin {lemma} let $b$ be a right extremal box which occurs in a right extremal column $C_i$ of $\mathscr D$ in row $R_j$.  Let $r$ be the entry of $b$.  Then in $\mathscr T'$, the entry $r$ exactly occurs in row $R_j$ and all the columns to the right of $C_i$.
  \end {lemma}

  \begin {proof} Indeed by Lemma \ref {4.3}, $b$ remains right extremal as columns are removed from the right until it itself is removed.
  \end {proof}

   \subsubsection {A Column Staircase}\label {4.6.2}

  Following Theorem \ref {3.6}(i), take $b:=b(\mathscr S)$ for some column staircase $\mathscr S = \{C'_{i_j}\}_{j=u}^c: \height C'_{i_j}=j$ with profile $i'_j:j \in [u-1,c]$. Recall the definition of $C'_{i_{u-1}}$ in \ref {3.2}.  It has height $u$ and $b(\mathscr S)$ is the box $R_u\cap C_{i'_{u-1}}$.

   \

   \textbf {Notation.} Let $C'_{i_{c+1}}$ be the last column which needs to be removed from the right of $\mathscr D$, so that $\mathscr S$ becomes right extremal (equivalently that $b=b(\mathscr S)$ becomes extremal) or let $C'_{i_{c+1}-1}$ be the last column of $\mathscr D$.

  \begin {lemma}  Let $r$ be the entry of $b(\mathscr S)$. Then in $\mathscr T'$, the entry $r$ exactly occurs in $C_{j}\cap R_{m+1}: j\in [i'_m,i'_{m+1}-1]:m \in [u-1,c]$, in $\mathscr T'$.

  \end {lemma}

  \begin {proof} This follows from Proposition \ref {3.5} as columns are removed from the right following the removal of $C'_{i_{c+1}-1}$.
  \end {proof}

   \subsubsection {Combining the two cases}\label {4.6.3}


    We may remove columns from the right until a given box becomes right extremal.  Then by Theorem \ref {3.6}, it is either some $b(\mathscr S)$ with $\mathscr S$ non-empty or an extremal box in an extremal column.  In the latter case we may view it as $b(\mathscr S)$ with $\mathscr S$ empty.  Then writing $b$ as $C'_{i_0}\cap R_{m+1}$ and adapting the notation of \ref {4.6.2}, the entry  $r$ of $b$ occurs in $C_{j}\cap R_{m+1}: j\in [i'_0,i'_1-1]$ in $\mathscr T'$.  Here the profile staircase is flat.

    \

  \textbf{Remarks}.
   Thus the repeated entries in $\mathscr T'$ follow the profile staircase of $\mathscr S^c_u$.

  Notice that by definition $C'_{i_m}:m \in [u,c]$ has height $m$ in $\mathscr D$.  Thus in $\mathscr T$ the box in $R_m\cap C'_{i_m}$ already has an entry whilst the box just below in $\mathscr D'$, that is in $R_{m+1}\cap C'_{i_m}$ does not.

\subsection {A Second rather Simple Algorithm for the Composition Tableau $\mathscr T'$} \label{4.7}

\

We now describe a different and very simple algorithm to describe $\mathscr T'$ which will be used exclusively in Section \ref{5}.  It will give the reader much needed relief! It makes no use of column staircases, though rather naturally the notion of a pair of neighbouring columns is still needed.

Recall the natural order $i\leq j$ on the set $[1,n]$ of entries of $\mathscr T$.

 Define a new total order $i'\preceq j'$ on these entries by \textit{defining} them to increase strictly down the columns and then from \textit{right to left}. Let $\{1',2',\ldots,n',\}$ be the set $[1,n]$ written in increasing order for $\preceq $.   This is not quite $\{1,2,\ldots,n,\}$ in reverse order.   In particular $1'$ is the first entry in the last column, generally different from $n$.

Let $\mathscr T(s)$ be the tableau starting from $\mathscr T$ and obtained after $s$ steps by placing all the entries $\{1',2',(l-1)'\}$ and some of the $l'$ in $\mathscr D'$ sequentially by the inductive procedure given below.

 A column $C_j$ is regarded as column of $\mathscr T$ and of $\mathscr T(s)$.  However $\height C_j$ will  \textit{always} refer to its height in $\mathscr T$, equivalently in $\mathscr D$.  Thus if $\height C_{j+1} < t$ then the box $b:=R_{t} \cap C_{j+1}$ is empty in $\mathscr T$, may yet have an entry in $\mathscr T(s)$.  Also, columns are regarded as neighbours only if they are neighbours in $\mathscr D$  (see example below).

  Let $R_t\cap C_j$ be the rightmost box of $\mathscr T(s) $ with entry $l'$.

  Suppose $\height C_{j+1} < t$.

  \

$(i)$. If $b$ is empty in $\mathscr T(s)$, put $l'$ in the box  $R_{t}\cap C_{j+1}$ of $\mathscr T(s)$ to obtain $\mathscr T(s+1)$.

\

$(ii)$. If not add no new entry $l'$ to  $\mathscr T(s)$.

\

  Suppose that $\height C_{j+1}=t$.

  \

  $(iii)$. If $C_{j+1}$  has a left neighbour in $\mathscr D$ and has no entry in  $R_{t+1}$ with respect to $\mathscr T(s)$, put $l'$ in $R_{t+1}\cap C_{j+1}$ to obtain $\mathscr T(s+1)$.

  \

  $(iv)$. Otherwise  add no new entry $l'$ to $\mathscr T(s)$.

  \

  \textbf{Definition}. If (ii) or (iv) apply, we say that $l'$ is stopped just before it reaches $C_{j+1}$.  Notice that $l'$ goes right and downward (though not necessarily strictly).  If at some point $l'$ enters $R_t$, we say it reaches $R_t$.

  \

  Let $\mathscr T(\infty)$ be the tableau obtained starting from $\mathscr T$ when the labels $(1',2',\ldots,n')$ have been inserted, possibly many times, into $\mathscr D'$ by the above rules.

\

\textbf{Examples.}  Consider the composition $(1,2,1)$. One has $4\preceq 2 \preceq 3\preceq 1$.  By (iii) with $j=2,t=1$, a second $2$ is put in $b:=C_3\cap R_2$.  Then by (ii), the entry $3$ is not placed in $b$. This exactly corresponds to $b(3)$ not being extremal and indeed it is linked to $b(4)$ by line with label $1$, whilst $b(2)$ is extremal.  Again although $C_2$ and $C_3$ both have height $2$ in $\mathscr T(s+1)$, they are \textit{not} considered as neighbours and so $(iii)$ is not applied to put a $3$ in $C_3\cap R_3$.  This would have made $b(3)$ extremal which is false.

 On the other hand if we omit the first column and keep the labelling of the remaining boxes then by $(i)$, $3$ not $2$ is placed in $b$ and $b(3)$, but not $b(2)$ is extremal.

   These observations lie behind the proof of the proposition below.
 \subsection {Semistandard Tableaux}\label{4.8}

 We call a tableau (with entries in $[1,n]$) semi-standard if entries strictly increase down the columns and decrease along the rows on going from left to right, with no gaps down columns.  (This might not be an accepted definition.)

\begin {lemma}  The tableau $\mathscr T(\infty)$ is semi-standard for the order relation $\preceq$.
\end {lemma}

\begin {proof} Obviously $\mathscr T$ is semi-standard by the definition of $\preceq$.  We prove the assertion for $\mathscr T(\infty)$ by induction adding one new integer $l'$ at a time.  Observe that by \ref {4.7}(i) at no step in the construction of $\mathscr T(\infty)$ is there a column with a gap, that is an empty box with entries above and below. Consider a box $b$ with entry $l'$ in $\mathscr T(s)$.

If there is no entry in a box $b'$ just to the right of $b$, we put $l'$ in $b'$, by \ref {4.7}(i).
By the induction hypothesis, the box above $b'$ and to the right of $b'$ has only entries than $\preceq l'$ , whilst the box below $b'$ is empty, since there are no gaps.

If $b'$ is already full but the box $b''$ just below $b'$ is empty, then we \textit{may} (resp.  \textit{may not}) be putting
$l'$ in $b''$, by \ref {4.7}(iii)(resp. \ref {4.7}(iv)).

As before the box to the right of $b''$ has an entry $\succeq l'$, whilst the box $b'$ above $b''$ also has an entry $\succeq l'$ because it is to the right of $b$.  Hence the induction carries forward.
\end {proof}

\textbf{Remark 1}.  In particular the entries in $\mathscr T(\infty)$ are strictly increasing down the rows for $\preceq$ whilst this is generally false for $\leq$.

\

\textbf{Remark 2}.  We may obtain a ``genuine'' semi-standard tableau from $\mathscr T(\infty)$ with entries in [1,n] also increasing for $\leq$ from left to right and strictly down columns as well as having no gaps in rows.  This is by taking the left-right mirror image of $\mathscr T(\infty)$ and replacing $i'$ by $i$.  It has as a sub-tableau, the tableau obtained by performing the same operations on $\mathscr T$.

\subsection {Equality of completed Tableaux.} \label{4.9}

\begin {prop}  $\mathscr T(\infty)=\mathscr T'$.
\end {prop}

\begin {proof}  Briefly the repeated entries in $\mathscr T(\infty)$ also follow the profile staircase of the staircases $\mathscr S^c_u$ - see the remarks in \ref {4.6.3}.

In mored detail we prove the assertion of the Proposition inductively using $(i)-(iv)$ of \ref {4.7} going through the columns from left to right.  It is obvious for the first column, so assume it holds up to $C_j$.  By \cite [Lemma 5.4.9]{FJ2} we may by successively removing columns from the right, assume that $C_{j+1}$ is the last column of $\mathscr D$. The assertion is then obvious for all the entries of $C_{j+1}$ which already appear in the common sub-tableau $\mathscr T$.

Take  $\height C_{j+1}=t$ (in $\mathscr D$).  Let $m'$ be the entry of $C_j\cap R_{t'}$ in $\mathscr T(\infty)$ and let $b$ be the unique box in $\mathscr T$ with label $m'$.

Take $t'>t+1$. Then by (i) above, $m'$ is the entry of $C_{j+1}\cap R_{t'}$ in $\mathscr T(\infty)$. By Theorem \ref {3.6}(ii) and Lemma \ref {3.5.2}, this also holds for $\mathscr T'$.


Take $t'=t+1$.  Suppose $C_{j+1}$ is part of a column staircase with the box $R_{t}\cap C_j$ having entry $l'$ in $\mathscr T(\infty)$. Then (Proposition \ref {3.5.3}) $b$ is linked by a line with label $1$ to $C_{j+1}\cap R_t$ and so is not extremal, so $m'$ does not appear in $C_{j+1}$ in $\mathscr T'$. Yet $l'\preceq m'$, so $l'$ enters $R_{t+1}\cap C_{j+1}$ by (iii) and so $m'$ is stopped from entering $R_{t+1}\cap C_{j+1}$ by (ii) and hence  $m'$ also does not appear in $C_{j+1}$ in $\mathscr T(\infty)$.  Conversely if $C_{j+1}$ is not part of a column staircase, $b$ is extremal in $\mathscr T'$ by Theorem \ref {3.6}(ii), so $m'$ does appear in $C_{j+1}$ in $\mathscr T'$.  Yet $l'$ does not enters $R_{t+1}\cap C_{j+1}$ by $(iv)$, leaving this place empty for $m'$ to enter by $(i)$. Thus $m'$ also  appears in $R_{t+1}\cap C_{j+1}$ in $\mathscr T(\infty)$.

Take $t'=t$. Suppose $C_{j+1}$ is part of a column staircase.  Then $m'$ enters $R_{t+1}\cap C_{j+1}$ in $\mathscr T(\infty)$ by (iii), whilst by Proposition \ref {3.6}, $m'$ also appears in $R_{t+1}\cap C_{j+1}$ in $\mathscr T'$.  Conversely if  $C_{j+1}$ is not part of a column staircase, $b$ is not extremal in $\mathscr T'$, so $m'$ does not appear in $C_{j+1}$ in $\mathscr T'$.  Yet by $(iv)$, $m'$ does not enter $R_{t+1}\cap C_{j+1}$, so $m'$ does not   appear in $C_{j+1}$ in $\mathscr T(\infty)$.

Take $t'<t$.  Then by (ii) and (iv), $m'$ is stopped from entering $C_{j+1}$ in $\mathscr T(\infty)$, whilst $b$ is not extremal, so does not appear in $C_{j+1}$ in $\mathscr T'$.

\end{proof}

\textbf{Examples}.  Consider the composition $(2,1,1,2,2)$.  The column staircases are $C_3,C_4$ and $C_4,C_5$.  Only the latter is extremal.  Indeed $b(3),b(8)$ are joined by a line with label $1$ by the gathering of loose ends \cite [5.4.7]{FJ2}.

Thus $\mathscr T'$ has $\mathscr T$ as a sub-tableau with $2$ also occurring in $C_2 \cap R_2$, $3$ also occurring in $C_3\cap R_2, C_4\cap R_3$ and $6$ also occurring in $C_5\cap R_3$.  One checks it coincides with $\mathscr T(\infty)$, noting that $6 \preceq 3$.

On the other hand for the composition $(2,1,1,2,1)$, both column staircases $C_2,C_4$ and $C_5$ are extremal.  Then $\mathscr T'$ differs from the previous case by having its last column with entries $\{7,5,3\}$ going downwards.  One checks it coincides with $\mathscr T(\infty)$.

%
%
%
%
%

\section {Reconstructing the Weierstrass Section} \label {5}

\subsection {Goal}\label {5.1}

Our aim is to use the composition tableau as presented by $\mathscr T(\infty)$ to show that the lines in $\mathscr D$ it defines, recover $e+V$ as a \textit {Weierstrass section without using the very difficult combinatorics of \cite [5.4] {FJ2}}.  This will be a great simplification. Indeed for this and apart from notation and definitions the reader will be happy to forget the previous sections and most the contents of our previous papers.  Indeed the reader is \textit{advised} to forget the construction of \cite [5.4]{FJ2} to avoid getting into a fearful muddle.

\subsection {The Strategy}\label {5.2}

Recall (\ref {2.4}) the composition map $t\mapsto r_t$.   One has $r_t \neq 0$, if $t \leq \height \mathscr D^k$, whilst $r_t = 0$, if $t > \height \mathscr D^k +1$.

 When $|\mathscr D^k| <  c_k$, the box in $\mathscr D^k$ with entry $r_t$ is joined to the box $C_k\cap R_t$ with a line labelled by a $1$.

In addition  \cite [Lemmas 3.4.2(i), 3.6.8]{FJ3} show how this description is modified when the above inequality fails.

Using this modification we may consider that the last column of $\mathscr T'$ determines the set $\ell(\mathscr D^k,C_k)$.

Then we repeat this procedure removing successively columns from the right of $\mathscr D$.

Then (compare to \cite [Lemma 5.4.9]{FJ2}) we have the

\begin {lemma}  $\ell(\mathscr D)$ is the disjoint union of $\ell(\mathscr D^k)$ and $\ell(\mathscr D^k,C_k)$.
\end {lemma}

\begin {proof}  By construction.
\end {proof}


\subsection {Stabilisation}\label {5.3}

\subsubsection {Above the Lowest Row of $C_k(s)$}\label {5.3.1}

Suppose  $t \in [1,s-1]$.  By \cite [Lemma 3.4.2 (i)]{FJ3} the entry of $\textbf{C}_k(s+1)\cap R_t$ is that of $\textbf{C}_k(s)\cap R_t$.  This exactly means that a line joins  the box $R_t \cap C_k(s)$ to the  box $b$ in $\mathscr D^k$ with entry $r_t$, and also joins the box $R_t \cap C_k(s+1)$ to this same box  $b$ and with the same labels.

All these left going lines are up-going.

\

\textbf{Conclusion.}  Thus from the composition map and backward induction we can recover the lines from $\mathscr D^k$ to all but the lowest row of $C_k(s)$ as $s$ decreases.

\

\subsubsection {The Lowest Row of $C_k$}\label {5.3.2}

For the computation of the lowest row of $C_k$, we recall \cite [3.6.8]{FJ3}, which we state in a slightly more precise form recalling our change of convention with respect to up and down-going compared to \cite {FJ3}.

\begin {lemma} Set $\height C_k=t$. Let $r_t$ be the image of $t$ under the composition map and $b(r_t)$ the  box in $\mathscr D^k$ with entry $r_t$.

\

$(i)$.  Suppose $t \in [\mathscr D^k]$.  Then the left going line joining  $C_k\cap R_t$ to $b(r_t)$ has label $\ast$.
It is up-going.

Suppose $r_{t+1}\neq 0$.  Then there is a left going line from  $C_k\cap R_t$ with label $1$ to $b(r_{t+1})$.

It may be down-going, but by at most one row.

\

$(ii)$. Suppose $t \notin [\mathscr D^k]$.  Then the left going line joining  $C_k\cap R_t$ to $b(r_t)$ has label $1$.

It is up-going.

\

$(iii)$.  There are no other left going lines from $C_k\cap R_t$ to $\mathscr D^k$.
\end {lemma}

\begin {proof}  Apply \cite [Lemma 3.6.8]{FJ3} and \cite [5.4.8(ix)] {FJ2} for the last part of (ii).

\end {proof}

\textbf{Remark.}  As promised in \ref {3.6} we have shown that the down-going property only slightly fails, namely in $(i)$.

\

\textbf{Examples}.  Suppose $\mathscr D$ is defined by the composition $(2,1)$.  In this $b(2),b(3)$ are extremal.  Thus $r_1=3,r_2=2,r_3=0$.  Suppose a column $C_3$ of height $3$ is adjoined to the right of $\mathscr D$.  When $b(6)$ is removed so that $\height C_3$ becomes $2$, the label $1$ on the line $\ell_{2,5}$ is replaced by a $\ast$, but no other line joins $b(5)$.   When in addition $b(5)$ is removed so that $\height C_3$ becomes $1$, the label $1$ on the line $\ell_{3,4}$ is replaced by a $\ast$ and $\ell_{2,4}$ is adjoined.  The latter is right and up-going by one line and has label $1$.

Now suppose $\mathscr D$ is defined by the composition $(m,m-1,\ldots,1,1,2,\ldots,m)$. The box $b:=b(\frac{m(m+1)}{2})$ is extremal.  Moreover in this case $m = \height \mathscr D$.  Suppose a column $C$ of height $m+1$ is adjoined to the right of $\mathscr D$.  Then there would be a line labelled by a $1$ joining $b$ and the box $C\cap R_{m+1}$ which is right and down-going by $m$ lines.  This translates to $r_{m+1}=1$.  One may note that $\mathscr T'$ has $\frac{m(m+1)}{2}$ in the $(m+1)^{th}$ in $R_{m+1}$ as well as in all entries of the diagonal joining the co-ordinates $((1,m), (m+1,2m))$, as expected from Proposition \ref {4.9}.

\subsection {Resulting Description of $\ell(\mathscr D^k,C_k)$ }\label {5.4}

We shall now view $\ell(\mathscr D)$ as being given by \ref {5.2}, the definition of $\mathscr T'$ and \textit{imposing} \ref {5.3} to describe $\ell(\mathscr D^k,C_k)$.  This gives in particular.

Set $s=|\mathscr D^k|$ and $t =\height C_k$.

\

$(i)$.  Suppose $s < t$.

  Then for all $i \in [1,s+1]$, the box $C_k\cap R_i$ has a unique left going line, except when $i=s+1$ and $r_{s+1}=0$. It has label $1$ and is up-going.

\

$(ii)$. Suppose $s>t$.

If $t \in [\mathscr D^k]$,
the box $C_k\cap R_t$ has a unique left and up-going line with label $\ast$ and a unique left going line with label $1$ which is at most down-going by one row.

If $t \notin [\mathscr D^k]$, then for all $i \in [1,t]$, the box $C_k\cap R_i$ has a unique left going line with label $1$ and no other left going lines.  It is up-going.

\

$(iii)$. Suppose $s=t$.  Then every box in $C_k\cap R^{t-1}$ has a unique left going line with label $1$. The box $C_k \cap R_t$ has a unique left going line with label $\ast$ and possibly a left going line with label $1$.  These lines are up-going and in this the last by the hypothesis $s=t$.

\

\textbf{N.B.} By Lemma \ref {5.2}, $C_k$ can be assumed to be \textit{any} column $C''$ of $\mathscr D$.

In this $C''=C_k$ in \ref {5.4} means that $\mathscr D^k$ consists of  all columns strictly to the left of $C''$ and computing $\mathscr T(\infty)$ relative to $\mathscr D^k$.

\

The composition tableau $\mathscr T'$ together with Lemmas \ref {5.2} and \ref {5.3.2} and \ref {5.4} must recover $\ell(\mathscr D)$.  We show below how many properties of $\ell(\mathscr D)$ given in \cite [5.4]{FJ2} are obtained directly and in a much easier fashion using $\mathscr T(\infty)$ and \ref {5.4}.

\subsection {Lines with label $1$}\label{5.5}

\

 \textbf{Notation}. Let $C,C'$ denote a pair of neighbouring columns of $\mathscr D$ having height $t$, with $C$ to the left of $C'$.  Let $R^t[C,C']$ be the set of boxes lying in $R^t$ between $C,C'$.   We shall write $]C,C']$ for the boxes strictly to the right of $C$, and so on.

  Recall \ref {5.4} (ii),(iii). Let $l_\ast$ denote the unique left going line from $C'\cap R_t$ with label $\ast$.

\subsubsection {Left Going Lines}\label {5.5.1}

 \begin {cor} A box $b$ in $R^t]C,C'[$ has a unique left going line labelled by a $1$.  Its left end point lies in $R^t[C,C'[$.

  \end {cor}

 \begin {proof}  Let $C''$ be the column containing $b$.  Since $C,C'$ are neighbours of height $t$ one has $\height C'' \neq t$.  Taking $C''=C_k$ in \ref {5.4}, $l$ must be up-going unless  $\height C''=t' \leq t-1$ and then  at most down-going by just one row. Thus its left end point $b'$ lies in $R^t$.

Recall \textbf{N.B.} of \ref {5.4}.

 If $b'$ lies strictly to the left of $C$, then the entry $s$ of $b'$ would have to appear below $R_{t+1}\cap C$ in $\mathscr T(\infty)$ and so $s$ could not appear in $R^t$ in any column of $\mathscr T(\infty)$ to the right of $C$.  This excludes the right end-point of $l$ lying above  $R_{t-1}$ in $C''$, if $\height C'' \leq t-1$, or above $C''\cap R_t$ if $\height C'' >t$.  Recalling the first paragraph, this gives a contradiction proving the Corollary.

 \end {proof}

 \subsubsection {Right Going Lines}\label {5.5.2}

 \begin {cor}  A box $b$ in $R^t[C,C'[$ has a unique right going line labelled by a $1$ with end point in $R^t]C,C']$, unless this line is $\ell_\ast$.
 \end {cor}

 \begin {proof} Let $r$ be the entry of $b$ in $\mathscr T$.  In $\mathscr T(\infty)$, $r$ is moved to the right through the rules of \ref {4.7}. If it stopped (via (ii),(iv)) just before some column $C''$, then $b$ is joined to a box in $C''$ .


  If in the transition $\mathscr T(s) \mapsto \mathscr T(s+1)$ the entry $r$ is not stopped, then by (i),(iii) of \ref {4.7}, it is moved rightwards and enters $R_{t'}$ for some $t'\leq t$. (This is because no column strictly between $C,C'$ can have height $t$).  Consequently at the first column $C''$  strictly between $C,C'$ of height $>t$, the entry $r$ is stopped.

  Finally let $C''$ be the column just to the left of $C'$.  If in $\mathscr T(\infty)$ one has $r \in C''\cap R_t$, then $b$ is joined to $C'\cap R_t$ by $l_\ast$.  Conversely the entry of $C''\cap R_t$ must also be an entry of some $b \in R^t[C,C'[$, because $r$ moves right and downwards and $\height C=t$.  Then $b$ is the left end-point of $l_\ast$. (One might add that in addition $r$ enters $ C'\cap R_{t+1}$, so $b$ is extremal in the diagram whose rightmost column is $C'$.)
 \end {proof}

 \textbf{Remark.} This calculus, though perhaps not quite at the level of the ladies who collected tickets in the London Underground, is certainly at the level of Oxford Undergraduates even those reading PPE\footnote {A soft option for Greats which traditionally trained the first class minds of England.} and in any case much easier than the torturous argument leading to \cite [5.4.8(v)]{FJ2}.

 \subsubsection {Addendum}\label {5.5.3}

 Independent of the last part of the above proof we have

 \begin {cor}  There is a unique left going line from $C'$ to a box $b$ in $R^t[C,C'[$.  It has label $1$ if $b\in R^{t-1}$ and has label $\ast$ if $b \in R_t$  (in the latter case the line is just $l_\ast$).
 \end {cor}

 \begin {proof}  Combine \ref {5.4} with Corollary \ref{5.5.2}.
 \end {proof}

 \textbf{Examples}.   If $r_{t+1}\neq 0$, then there is a left going line $l$ from $C\cap R_t$. However either the left hand end-point of $l$ is strictly below $R_t$ (as for the composition $(1,2,1)$), or strictly to the left of $C'$.
 In the latter case $l$ must ``hop over'' $C'$.  If $b \in R^t$, this can only happen if $C'$ has a left neighbour for example in the case $(1,1,1)$ or the case $(2,1,1,2,2)$.

In \cite [5.4.7]{FJ2}, it was said that $l$ was obtained by the ``gathering of loose ends'' - another complication of \cite [5.4.7]{FJ2} which an automatic consequence of our present construction.  This ``hopping over'' will be further illustrated in \ref {5.6.3}.

\subsection {Lines with Label $\ast$}\label{5.6}

\subsubsection {Left going Lines}{\label {5.6.1}

\begin {cor}  A box in $R^t_{C,C'}$ has at most one left going line with label $\ast$.  This occurs exactly when it lies at the bottom of a column $C''$ of height $t' \leq t$ having a left neighbour in $\mathscr D$.  Moreover its left end point lies in $R^t_{C,C'}$.
\end {cor}

\begin {proof} This first part is immediate from \ref {5.4}(ii),(iii).  The second part follows as in the proof of the second part of Corollary \ref {5.5.1}.
\end {proof}

\subsubsection {Right going Lines}{\label {5.6.2}

Recall (Definition \ref {3.2}) the notion of a column staircase $\mathscr S_u^c$ and the notation $b(\mathscr S^u_c)$. The last column in $\mathscr S^u_c$ (which has height $c$) will be denoted by $C'(c)$.

\begin {lemma} A box $b \in \mathscr D$
has a right going line labelled by a $\ast$ if and only if it is $b(\mathscr S^c_u)$  of a non-empty column staircase $\mathscr S^c_u$ and then it has $c-u+1$ such lines successively going down by $0,1,2,\ldots,c-u$ rows.

Let $C''(c+1)$ (resp. $C''(c)$) be the first column strictly to the right of $C'(c)$ of height $>c$ (resp. $c$) if such exists.  Then $b$ has a right going line labelled by $1$ going down by $c-u+1$ (resp. $c-u$) rows to $C''(c+1)\cap R_{c+1}$ (resp. $C''(c) \cap R_c$).
\end {lemma}

\begin {proof} Recall the notation of \ref {3.2}. By definition $\mathscr S^c_u$ consists of columns $C'_{i}$ of heights $i=u,u+1,\ldots,c$ each with a left neighbour $C_i$ and with no columns of height $\geq i$ between $C'_{i},C'_{(i+1)}:i<c$. The latter shorter columns can be ignored because they play no role in the subsequent calculation of $\mathscr T(\infty)$.
Let $C_{(u-1)}$ be the unique rightmost column of height $\geq u$ to the left of $C'_{u}$, possibly $C_u$, though generally to its right.  Let $r$ be the entry in the box $C_{(u-1)}\cap R_u$, which by definition is just $b:=b(\mathscr S^u_c)$.

Recall the rules $(i)-(iv)$ of \ref {4.7} determining $\mathscr T(\infty)$. By these, $C_{(i-1)'}\cap R_{i}$ acquires the entry $r$ for all $i=u,u+1,\ldots,c$.  By Proposition \ref {4.8}, and \ref {5.4} applied to $C_{i'}$, we obtain a line, labelled by a $\ast$,  from $b$ to the box $b_{i}$ in $C_{i'}\cap R_{i}$.  It is right and down going by $i-u$ rows: $i=u,\ldots,c$. Again by \ref {5.4} applied to $C(c)$ we obtain a line labelled by a $1$ to $C''(c+1)\cap R_{c+1}$, if it exists or otherwise to $C''(c)\cap R_c$, if it exists, or no line at all.

\end {proof}

\textbf{N.B.}  If $b \in R^t(C,C')$ and $c+1\leq t$ then either $C''(c+1)$ or $C''(c)$ automatically exist!

\subsubsection {Adjacent Column Staircases}{\label {5.6.3}

We need to examine the case when there are adjacent non-empty column staircases, particularly when only the leftmost one has height $>1$. Thus fix $u \in \mathbb N^+$ and consider a right going sequence of column staircases $\mathscr S_u^u(i):i=1,2,\ldots, m,\mathscr S_u^c(m+1)$.  As before we may ignore the columns of height $<u$.  In doing so we obtain $m$ adjacent columns $C'_{i,u}:i=1,2,\ldots, m$ of height $u$
all having left neighbours $C_{i,m}$ and followed by columns $C'_{m+1,j}$ of heights $j=u,u+1,\ldots,c$, also each admitting left neighbours $C_{m+1,j}$.

Let us compute some relevant entries of $\mathscr T(\infty)$.

Let $C'_{0,u}$ be the unique rightmost column of $\mathscr D$ of height $u'\geq u$ strictly  to the left of $C'_{1,u}$ (possibly $C_{1,u}$ itself) and $r_j$ the entry of the box $C'_{j,u} \cap R_u:j=0,1,2,\ldots,m$.  In the inductive construction of $\mathscr T(\infty)$, this entry $r_0$ first enters $C'_{1,u}\cap R_{u+1}$, but is then stopped by the entry $r_1$ of
$C'_{1,u} \cap R_u$ which enters $C'_{2,u}\cap R_{u+1}$, and so on, that is the entry $r_j$ of $C'_{j,u} \cap R_u$  enters $C'_{j+1,u}\cap R_{u+1}$, for all $j=0,1,2,\ldots,m-1$.  Following this, $r_m$ enters the box $C'_{m+1,j}\cap R_{j+1}$ for all $j=u,u+1,\ldots,c$.

Through \ref {5.4} we obtain a horizontal line $l_i$ joining the box in  $C'_{i-1,u}\cap R_u$ to $C'_{i,u}\cap R_u$, for all $i=1,2,\ldots,m$ and a down-going line $l'_i$ joining the box $C'_{m,u}\cap R_u$ to $C'_{m+1,j}\cap R_{j}$, for all $j=u,u+1,\ldots,c$, all labelled by $\ast$ and possibly in addition (see last part of \ref {5.6.2})  a down-going line from $C'_{m,u}\cap R_u$ with label $1$ to $R_{c+1}$ or to $R_c$. In addition \ref {5.4} gives $m$ horizontal lines $l''_k:k=0,1,2,\ldots, m-1$ joining $C'_{m+1-k,u},C'_{m-1-k,u}$ in row $R_u$ labelled by a $1$.

\

\textbf{Example.}   Consider the composition $(1,2,1,1,1,2,3)$.  Then the boxes with the pairs of entries $(2,4),(4,5),(5,6),(5,8)$ are joined by a line labelled by $\ast$ and  the boxes with the pairs of entries $(2,5),(4,6),(5,11)$ are joined by a line labelled by $1$, by the above.   There are also lines labelled by $1$ joining the pairs $(1,2),(3,4),(6,7),(7,9),(8,10)$.  All this is ridiculously simple and the general case is barely more complicated.

\subsection {Composite Lines}\label {5.7}

Retain the above hypothesis and notation.

Recall \cite [4.2.3]{FJ1} that a line is said to be composite if it is a concatenation of lines joining boxes going from left to right.  Two composite lines are said to be disjoint if they do not pass through a common box.

\begin {prop}  There is a unique disjoint union of composite lines passing though every box in $R^t[C,C']$ formed from the unique left-going line $l_\ast$ from $C\cap R_t$ with label $\ast$ together with all the lines with label $1$ joining boxes in $R^t[C,C']$.
\end {prop}

\begin {proof} By Lemma \ref {5.2}, we can assume $C'=C_k$, without loss of generality.  Set $\height C_k=t$.  By Lemma \ref {5.6.3} the left going line from $C'\cap R_t$ with label $1$ if it exists, has end-point outside $R^t[C,C']$ and so can be ignored.  By Corollary \ref {5.5.1} there is a unique concatenation of left going lines with label $1$ from $C\cap R_{t'} :
t' \in [1,t]$, starting with $l_\ast$ if $t'=t$ and ending in $C$  or hopping over $C$.  (Eventually a counting argument shows  that no hopping over $C$ takes place.)

  By uniqueness of  left going lines with label $1$ from a given box and the uniqueness of right going lines in Corollary \ref {5.5.2}, the resulting composite lines are disjoint and of those which meet $C$ there are at most $t$ and exactly $t$ if no hopping over $C$ takes place.

On the other hand by Corollary \ref {5.5.2} also asserts the existence of a unique right going line from any box in $R^t[C,C'[$ to a box in $R^t]C,C']$ labelled by a $1$ unless it be the line $l_\ast$ (which meets $C\cap R_t$).

We conclude that there are exactly $t$ disjoint composite lines passing through every box in $R^t[C,C']$ of which exactly one line has label $\ast$.  In addition no hopping over $C$  by the lines, which comprise these disjoint composite lines, can take place.  Finally uniqueness of these lines with label $1$ implies that all the lines in $R^t(C,C')$  must be used!
\end {proof}

\subsection {The General Case}\label {5.8}

We strengthen the above result retaining the same notation.

\begin {thm}  There is a unique disjoint union $\mathscr U^t[C,C']$ of composite lines passing through all the boxes in $R^t[C,C']$.
\end {thm}

\begin {proof}

The theorem follows from Proposition \ref {5.7} if we can show in  \textit{any} disjoint union $\mathscr U^t[C,C']$ of composite lines passing through all the boxes in $R_{C,C'}^t$, that $l_\ast$ is the only line in the family $\ell^t[C,C'](\mathscr F)$ of lines forming these composite lines, which is labelled by a $l_\ast$. (This was the point of view we adopted in \cite [4.4.11]{FJ1} and \cite [5.4]{FJ2}.)

To exclude all such lines labelled by an $\ast$ excepting $l_\ast$ from  $\mathscr U^t[C,C']$ we prove that any left going line $l \in \mathscr U^t[C,C']$ with right end point in $b\in R_{c+1}$ is labelled by a $1$, excepting $l_\ast$, by downward induction on $c$.  \textit{Then any other left going line to $b$ (necessarily labelled by a $\ast$ through Corollary \ref {5.5.1}) is excluded becuase the union $\mathscr U^t(C,C')$ is required to be \textbf{disjoint}}. Here we can assume $c+1 \leq t$.


Now by Corollary \ref {5.6.2} every line labelled by a $\ast$ has a left end point which is some $b(\mathscr S)\in R^c(C,C')$.  Thus all these lines are excluded as long as the line with label $1$ in the conclusion of this corollary are strictly down-going.

However it can happen that such a line is horizontal.  This situation is handled by \ref {5.6.3}.  Indeed in the notation of \ref {5.6.3}, the horizontal lines $l_i:i=1,2,\ldots,m$ labelled by $\ast$ are successively excluded.  Indeed $l_1$ is excluded by the strictly down-going line labelled by a $1$ (or by $l_\ast$) to $C_{m-1}$.  This forces $l''_1$ to belong to $\mathscr U^t[C,C']$ which in turn excludes $l_2$ and so on.

Thus the assertion of the theorem results.

\end {proof}

As Sherlock Holmes might have said.   Elementary my dear Watson!

\

\textbf{Example.}  In the example following \ref {5.6.3}, taking $t=2$, the line $l_{5,8}$ necessarily belonging to $\mathscr U^t[C,C']$ eliminates $l_{5,6}$ forcing $l_{4,6}\in \mathscr U^t[C,C']$ which in turn eliminates $l_{4,5}$ in turn forcing $l_{2,5} \in \mathscr U^t[C,C']$ in turn eliminating $l_{2,4}$. The disjoint union of composite lines joining $C,C'$ joins the boxes $b(2),b(5),b(8)$ and $b(3),b(4),b(6),b(7)$.  Easy!   The general case is barely more complicated.

\subsection {The existence of a Weierstrass Section}\label {5.9}

As in \cite [5.4.12]{FJ2} the above theorem implies that $(P_1,P_2)$] of cite [5.2] {FJ1} are satisfied by $\ell(\mathscr D)$.  As in \cite [Lemma 4.2.5(i,iii)]{FJ1} these further imply that $e+V$, as defined by $\ell(\mathscr D)$, is a Weierstrass section.  One may remark that $\ell(\mathscr D)$ selects a particular monomial in Benlolo-Sanderson invariant defined by the pair $C,C'$ of neighbouring columns.

We remark that by \cite [6.4.10]{FJ2} (proved in \cite [4.5.3]{FJ3} after much work), the pair $e+V$ determines uniquely an explicitly determined component of the zero fibre $\mathscr N$ containing $e$, namely $\mathscr N^e$.  In general this component does not have a dense $P$ orbit \cite [6.10.7]{FJ2}.  In particular it is not clear if $e$ itself determines this component or can lie in several components of $\mathscr N$.

Obviously an element of $e+v:v\in V$ is non-vanishing on some homogeneous invariant of positive degree if $v\neq 0$ and so $\overline{P(e+V)}\cap \mathscr N =\overline{Pe}$.


\subsection {The Canonical Component $\mathscr N^e$}\label {5.10}

Following \cite [4.4.2]{FJ3} we call $(i,j,k,l)$ a VS quadruplet if $\ell_{i,j},\ell_{k,l}$ (resp. $\ell_{j,k}$) are lines in $\ell(\mathscr D)$ with label $1$ (resp. $\ast$). They are easy to write down and of course determined by the pair $e,V$.

We can enlarge $e$ to $e_{VS}$ by joining to it the co-ordinates $x_{j,l}$ for every VS quadruplet $(i,j,k,l)$. We remark that $e_{VS}+V$ is still a Weierstrass section \cite [4.5.4]{FJ3}. Let $E_{VS}$ be the vector subspace of $\mathfrak m$ spanned by the root vectors in $e_{VS}$.  Miraculously $E_{VS} \subset  \mathscr N^e$ (combine \cite [6.9.8]{FJ2}, \cite [4.5.2]{FJ3}) and even more miraculously
\cite [4.5.3(iii)]{FJ3} the closure of $PE_{VS}$ equals $\mathscr N^e$, whilst this generally fails for $Pe_{VS}$.  Indeed the former fact was used to prove that $\mathscr N^e$ may have no dense $P$ orbit \cite [6.10.7]{FJ2}, without using a super/quantum computer, or even any computer at all, as would be needed for a rank computation.

\section {Reading the Composition Tableau}\label {6}

\subsection {Steps}\label {6.1}

For all $t \in [1,n]$, we retain our convention that $b(t)$ is the unique box in $\mathscr T$ labelled by $t$.

Let $\mathfrak s(t)$ denote the set of boxes in $\mathscr T(\infty)$ whose entries are $t$.  It is a profile staircase (see Proposition \ref {4.9}).
Define a step of $\mathfrak s(t)$ to be an entry of the box in $\mathscr T(\infty)$ just below where $t$ enters going down by one row.  Let $\mathfrak s_S(t)$ denote the set of steps of $\mathfrak s(t)$.  It may be empty.

It follows from the rules (ii),(iii) of \ref {4.7} that a step is an entry of $\mathscr T$ and hence the $\mathfrak s_S(t):t \in [1,n]$ are pairwise distinct.

\subsection {Lines with label $\ast$}\label {6.2}

From the observations in \ref {6.1} we obtain

\begin {lemma}

\

$(i)$.  For all $t \in [1,n]$, the right going lines from $b(t)$ with label $\ast$ meet the $b(s): s \in \mathfrak s_S(t)$.

\

$(ii)$.  For all $s\in [1,n]$, there is at most one left going line from $b(s)$ with label $\ast$.  It exists and meets $b(t)$ when $s \in \mathfrak s_S(t)$.
\end {lemma}


\subsection {D\'ebut et Fin}\label {6.3}

Take $t \in [1,n]$.  Recall that in the profile staircase $\mathfrak s(t)$, each entry $t$ occurs in at most one column and the columns in which $t$ occur are consecutive.

Define the beginning (d\'ebut) (resp. end (fin)) of the profile staircase $\mathfrak s(t)$ to be the leftmost (resp. rightmost) box $b^d(t)$ (resp. $b^f(t)$) in which $t$ occurs in $\mathfrak s(t)$.  One may remark that $b^d(t)$ is the unique box $b(t)$ in $\mathscr T$ with entry $t$.  Let $C'$ be the column (resp. row $R_u$) in which $b(t)$ occurs in $\mathscr T$ and $C$ the preceding column.

Let $R_v$ denote the row in which $b^f(t)$ occurs. Suppose in the column $C_{b(t)}$  just to the right of that containing $b^f(t)$, there is a box which occurs in $\mathscr D$ either in $R_{v-1}$ or in $R_v$ [if rule $(ii)$ (resp. $(iv)$) of \ref {4.7} applies]. Let $s$ be its entry.

\subsection {Lines with label $1$}\label {6.4}

Retain the  notation of \ref {6.3}.

\begin {lemma}

\

$(i)$.  For all $t \in [1,n]$, there is at most one right going line with label $1$ from $b(t)$ and exactly one if $C_{b(t)}$ exists, in which case its right end point is $b(s)$.

\


$(ii)$.   For all $t \in [1,n]$, there is at most one left going line $\ell$ with label $1$ from $b(t)$ and exactly one in the following two cases.

\

$(a)$. Either the entry $s\in C\cap R_u$  does not occur (in $\mathscr T(\infty)$) in $R_{u+1}\cap C'$.  In this case the left end-point of $\ell$ is $b(s)$.

\

$(b)$. Or if $s$ does occur in $\mathscr T(\infty)$ in $R_{u+1}\cap C'$ and there is an entry $s'$ of $\mathscr T(\infty)$ in $R_{u+1}\cap C$ In this case the left end-point of $\ell$ is $b(s')$.
\end {lemma}

\begin {proof} (i) follows from the remarks in \ref {6.3} and (ii) is its ``opposite''.  Of course these assertions are more difficult to state correctly than to prove.
\end {proof}

\subsection {Diagrammatic Presentation}\label {6.5}

Take $t \in [1,n]$.  As noted in \ref {6.1} each step is a box $b$ in $\mathscr D$.  By Lemma \ref {6.2} we may draw a vertical line in $\mathscr T(\infty)$ going down from $b$ by one row with label $\ast$ to represent the corresponding line in $\mathscr T$ with label $\ast$.  To this we adjoin the left (resp. right) going line labelled by a $1$ from $b^{d}(t)$ (resp. $b^{f}(t)$) to obtain a diagrammatic presentation of $\ell(\mathscr D)$ in $\mathscr T(\infty)$.

Now consider such a vertical line going from entry $j$ up to entry $k$. If $j$ (resp. $k$) has a left (resp. right) going line with label $1$ and end-point $i$ (resp. $l$), then $(i,j,k,l)$ is a VS quadruplet.  In this case join $j,l$ by a dotted line with label $1$. It is immediate that

\begin {lemma} The set of all lines with label $1$ in $\mathscr T(\infty)$ give the root subspaces of $\mathfrak m$ which define $E_{VS}$.
\end {lemma}

\textbf{Remark.}  In this diagrammatic presentation the elimination of the lines with label $\ast$ (excepting $l_\ast$) in proving Theorem \ref {5.8} takes place from right to left.  This makes the proof more transparent though perhaps not easier.

\

\textbf{Acknowledgements}.  During the academic year $2021-22$, when this paper was conceived and written, the first author was supported by ISF grant 2017/2019 as a postdoctoral student of Anna Melnikov at the University of Haifa.  We would lke to thank her for her generosity and hospitality.

\section {Index of Notation}\label {7}

Symbols used frequently are given below in the sections in which they are first defined.

$1.1$.  \   $I, \mathscr N$.

$1.2$.   \  $e,h,\mathfrak p,V$.

$1.3$.  \   $P, \mathfrak m$.

$1.4$.   \  $n, \textbf{M}, \mathscr N^e, E_{VS}$.

$2.1$.   \  $[u,v], \textbf{c}, \mathscr D, C_i, c_i, [\mathscr D], \mathscr D|, \mathscr T, R_j,R^i, b_{i,j}, x_{i,j}, b(t)$.

$2.2$.   \  $\ell(\mathscr D), \ell_{i,j}, L, \textbf{B}_i, \textbf{C}_i$.

$2.3$.   \   $\mathscr D^k, \ell(\mathscr D^k,C_k)$.

$2.4$.   \  $\ell_t, r_t$.

$3.5$.  \   $c^i, C_i^L, b_0,b_m$.

$4.2$.  \   $\mathscr T',\mathscr D'$.

$4.7$.  \   $\leq, \preceq,i', \mathscr T(s), \mathscr T(\infty)$.

$5.2$.   \  $\ell(\mathscr D^k), C_k$.

$5.5$.  \   $R^t(C,C'), \ell_\ast$.

\section {A Worked Example.}\label {8}

Consider the parabolic given by the composition  $(1,2,4,3,2,3,4,1,1,2)$.


   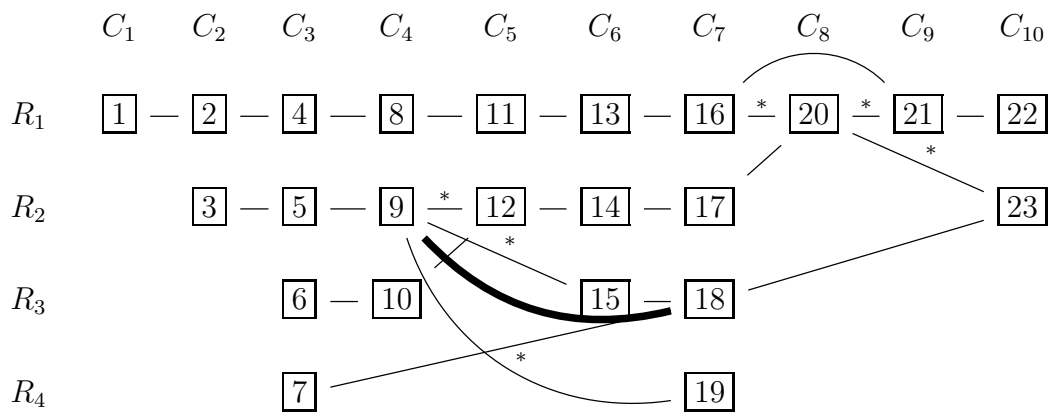
\begin{figure}[H]

\begin{center}
\begin{tikzcd}[row sep=1 em,
column sep = 0.9em]
  &C_1&C_2&C_3&C_4&C_5&C_6&C_7&C_8&C_9&C_{10}\\
R_1& \fbox {1}\arrow[-,r]&  \fbox {2}\arrow[-,r]& \fbox {4}\arrow[-,r]& \fbox {8}\arrow[-,r]& \fbox {11}\arrow[-,r]& \fbox {13}\arrow[-,r]& \fbox {16}\arrow[-,r,"*"]  \arrow[-,rr,bend left=45]& \fbox {20}\arrow[-,r,"*"]\arrow[-,drr,"*"]&\fbox {21}\arrow[-,r]&\fbox {22}\\
R_2&  &  \fbox {3}\arrow[-,r]& \fbox {5}\arrow[-,r, 
]& \fbox {9}\arrow[-,r,"*"]\arrow[-,drr,"*"]\arrow[-,rrrd,,bend right=3
0,
,line width=1mm ]\arrow[-,rrrdd,"*",bend right=40]& \fbox {12}\arrow[-,r]& \fbox {14}\arrow[-,r]& \fbox {17}\arrow[-,ur]  & &&\fbox {23}\\
R_3&  &   & \fbox {6}\arrow[-,r ]& \fbox {10}\arrow[-,ur]& & \fbox {15}\arrow[-,r, 
]& \fbox {18} \arrow[-,rrru] & &&\\
R_4&  &   & \fbox {7}\arrow[-,rrrru ]& & & & \fbox {19} & &&\\

\end{tikzcd}\\

%
%
%
%
\caption{The tableau $\mathscr T$ for this composition numbered according the procedure given in \ref {2.1}.  The labelled set of lines $\ell(\mathscr D)$ is given by the procedure described in \cite [5.4]{FJ2}. The set of $VS$ pairs $(i,j,k,l)$ is through their definition in bijection with the subset of lines $\ell_{j,k}$ labelled by a $\ast$ for which there exist lines $\ell_{i,j},\ell_{k,l}$ labelled by a $1$. Thus there are at most $6$ and in fact only $3$, namely $(5,9,12,14),(5,9,15,18),(17,20,21,22)$. Only the middle one is not superfluous and as such we call it ``bad'' - see \cite [4.4.3]{FJ3}. It gives an additional contribution to $e$ indicated by the boldface line. The classification of the superfluous quadruplets is unknown.  } \label{fig1}
\end{center}
\end{figure}

\begin{figure}[H]

\begin{center}
\begin{tikzcd}[row sep=1 em,
column sep = 0.9em]
  &C_1&C_2&C_3&C_4&C_5&C_6&C_7&C_8&C_9&C_{10}\\
R_1& \fbox {1}\arrow[-,r]&  \fbox {2}\arrow[-,r]& \fbox {4}\arrow[-,r]& \fbox {8}\arrow[-,r]& \fbox {11}\arrow[-,r]& \fbox {13}\arrow[-,r]& \blue \textbf{\fbox {16}} & \blue\textbf{\fbox {20}} \arrow[d,-,dashed,"*"]
 &\fbox {21}\arrow[d,-,dashed,"*"]
  \arrow[-,r]&\fbox {22}\\
R_2&  &  \fbox {3}\arrow[-,r]& \fbox {5}\arrow[-,r, 
]&  \textbf{\blue{ \fbox {9}}}& \fbox {12} \arrow[d,-,dashed,"*"]
 \arrow[-,r]& \fbox {14}\arrow[-,r]& \fbox {17}\arrow[-,ur]  &\textbf{\blue{ \fbox {16}}} \arrow[-,ru]&\textbf{\blue{ \fbox {20}}}\arrow[-,ru,bend right=30,line width=1mm] &\fbox {23}\arrow[d,-,dashed,"*"]\\
 R_3&  &  & \fbox {6}\arrow[-,r]& \fbox {10}\arrow[-,ru]&\textbf{\blue{ \fbox {9}}}\arrow[-,ru ,bend right=30 , line width=1mm] & \fbox {15}\arrow[d,-,dashed,"*"]\arrow[-,r, 
 ]&  \textbf{\blue{ \fbox {18}}}& \textbf{\blue{ \fbox {18}}} & \textbf{\blue{ \fbox {18}}}\arrow[-,ru]&\textbf{\blue{ \fbox {20}}}\\
  R_4&  &  & \textbf{\blue{ \fbox {7}}}&  \textbf{\blue{ \fbox {7}}}& \textbf{\blue{ \fbox {7}}}\arrow[-,ur]& \textbf{\blue{ \fbox {9}}}\arrow[-,ru,,bend right=3
0,
,line width=1mm]& \textbf{\blue\fbox {19}}\arrow[d,-,dashed,"*"] & \textbf{\blue{ \fbox {19}}} & \textbf{\blue{ \fbox {19}}} & \textbf{\blue{ \fbox {19}}}\\
    R_5&  &  && & & & \textbf{\blue{ \fbox {9}}} & \textbf{\blue{ \fbox {9}}}& \textbf{\blue{ \fbox {9}}}& \textbf{\blue{ \fbox {9}}}\\

\end{tikzcd}\\

\caption{The tableau $\mathscr{T}(\infty)$ obtained by the procedure given in \ref {4.7}.  The repeated entries are shown in blue (bold when reproduced in black and white). 
Its diagrammatic presentation follows \ref {6.5}. Because of the repeated entries the lines with label $\ast$ may be drawn vertically (see \ref {6.5}).  The three $VS$ pairs are indicated in boldface following \ref {6.5}.   The remaining lines recover those show in Figure $1$ with just one of the boldface lines incorporated.}

\end{center}
\end{figure}
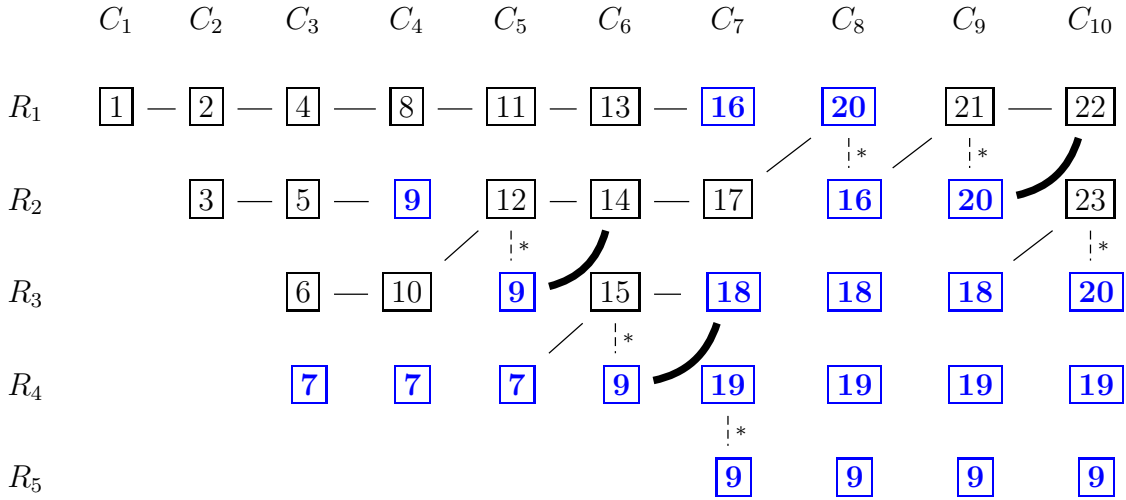

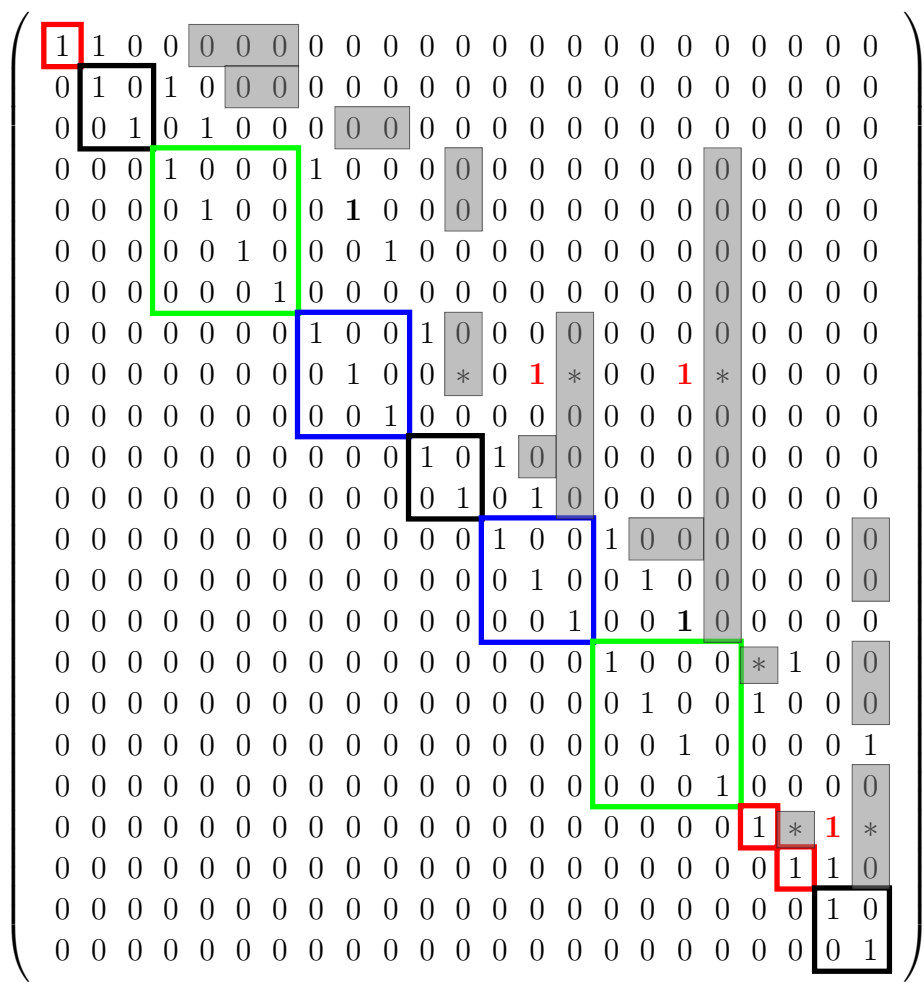
\begin{figure}[H]
\begin{center}
\begin{tikzpicture}
 \matrix [matrix of math nodes,left delimiter=(,right delimiter=)] (m)
 {
1 & 1 & 0 & 0 & 0 & 0 & 0 & 0 & 0 & 0 & 0 & 0& 0 & 0 & 0 & 0 & 0 & 0 & 0 & 0 & 0 & 0 & 0 \\
0 & 1 & 0 & 1 & 0 & 0 & 0 & 0 & 0 & 0 & 0 & 0& 0 & 0 & 0 & 0 & 0 & 0 & 0 & 0 & 0 & 0 & 0\\
0 & 0 & 1 & 0 & 1 & 0 & 0 & 0 & 0 & 0 & 0 & 0& 0 & 0 & 0 & 0 & 0 & 0 & 0 & 0 & 0 & 0 & 0 \\
0 & 0 & 0 & 1 & 0 & 0 & 0 & 1 & 0 & 0 & 0 & 0& 0 & 0 & 0 & 0 & 0 & 0 & 0 & 0 & 0 & 0 & 0 \\
0 & 0 & 0 & 0 & 1 & 0 & 0 & 0 & \mathbf{1} & 0 & 0 & 0& 0 & 0 & 0 & 0 & 0 & 0 & 0 & 0 & 0 & 0 & 0 \\
0 & 0 & 0 & 0 & 0 & 1& 0 & 0 & 0 & 1 & 0 & 0& 0 & 0 & 0 & 0 & 0 & 0 & 0 & 0 & 0 & 0 & 0 \\
0 & 0 & 0 & 0 & 0 & 0 & 1 & 0 & 0 & 0 & 0 & 0& 0 & 0 & 0 & 0 & 0 & 0 & 0 & 0 & 0 & 0 & 0 \\
0 & 0 & 0 & 0 & 0 & 0 & 0 & 1 & 0 & 0 & 1 & 0& 0 & 0 & 0 & 0 & 0 & 0 & 0 & 0 & 0 & 0 & 0 \\
0 & 0 & 0 & 0 & 0 & 0 & 0 & 0 & 1 & 0 & 0 & \ast & 0 &  \mathbf{\red{1}} & \ast & 0 & 0 & \mathbf{ \mathbf{\red{1}}} & \ast & 0 & 0 & 0 & 0 \\
0 & 0 & 0 & 0 & 0 & 0 & 0 & 0 & 0 & 1 & 0 & 0& 0 & 0 & 0 & 0 & 0 & 0 & 0 & 0 & 0 & 0 & 0 \\
0 & 0 & 0 & 0 & 0 & 0 & 0 & 0 & 0 & 0 & 1 & 0& 1 & 0 & 0 & 0 & 0 & 0 & 0 & 0 & 0 & 0 & 0 \\
0 & 0 & 0 & 0 & 0 & 0 & 0 & 0 & 0 & 0 & 0 & 1& 0 & 1 & 0 & 0 & 0 & 0 & 0 & 0 & 0 & 0 & 0\\
0 & 0 & 0 & 0 & 0 & 0 & 0 & 0 & 0 & 0 & 0 & 0& 1 & 0 & 0 & 1 & 0 & 0 & 0 & 0 & 0 & 0 & 0\\
0 & 0 & 0 & 0 & 0 & 0 & 0 & 0 & 0 & 0 & 0 & 0& 0 & 1 & 0 & 0 & 1 & 0 & 0 & 0 & 0 & 0 & 0 \\
0 & 0 & 0 & 0 & 0 & 0 & 0 & 0 & 0 & 0 & 0 & 0& 0 & 0 & 1 & 0 & 0 &  \mathbf{1} & 0 & 0 & 0 & 0 & 0 \\
0 & 0 & 0 & 0 & 0 & 0 & 0 & 0 & 0 & 0 & 0 & 0& 0 & 0 & 0 & 1 & 0 & 0 & 0 & \ast & 1 & 0 & 0 \\
0 & 0 & 0 & 0 & 0 & 0 & 0 & 0 & 0 & 0 & 0 & 0& 0 & 0 & 0 & 0 & 1 & 0 & 0 & 1 & 0 & 0 & 0 \\
0 & 0 & 0 & 0 & 0 & 0 & 0 & 0 & 0 & 0 & 0 & 0& 0 & 0 & 0 & 0 & 0 & 1 & 0 & 0 & 0 & 0 & 1 \\
0 & 0 & 0 & 0 & 0 & 0 & 0 & 0 & 0 & 0 & 0 & 0& 0 & 0 & 0 & 0 & 0 & 0 & 1 & 0 & 0 & 0 & 0 \\
0 & 0 & 0 & 0 & 0 & 0 & 0 & 0 & 0 & 0 & 0 & 0& 0 & 0 & 0 & 0 & 0 & 0 & 0 & 1 & \ast &  \mathbf{\red{1}} & \ast \\
0 & 0 & 0 & 0 & 0 & 0 & 0 & 0 & 0 & 0 & 0 & 0& 0 & 0 & 0 & 0 & 0 & 0 & 0 & 0 & 1 & 1 & 0 \\
0 & 0 & 0 & 0 & 0 & 0 & 0 & 0 & 0 & 0 & 0 & 0& 0 & 0 & 0 & 0 & 0 & 0 & 0 & 0 & 0 & 1 & 0 \\
0 & 0 & 0 & 0 & 0 & 0 & 0 & 0 & 0 & 0 & 0 & 0& 0 & 0 & 0 & 0 & 0 & 0 & 0 & 0 & 0 & 0 & 1\\
};
\draw[red,line width = 2pt] (m-1-1.south west) rectangle (m-1-1.north east);
\draw[line width = 2pt] (m-3-2.south west) rectangle (m-2-3.north east);
\draw[
green ,line width = 2pt] (m-7-4.south west) rectangle (m-4-7.north east);
\draw[blue,line width = 2pt] (m-10-8.south west) rectangle (m-8-10.north east);
\draw[line width = 2pt] (m-12-11.south west) rectangle (m-11-12.north east);
\draw[
blue,line width = 2pt] (m-15-13.south west) rectangle (m-13-15.north east);
\draw[green ,line width = 2pt] (m-19-16.south west) rectangle (m-16-19.north east);
\draw[red ,line width = 2pt] (m-20-20.south west) rectangle (m-20-20.north east);
\draw[red, line width = 2pt] (m-21-21.south west) rectangle (m-21-21.north east);
\draw[line width = 2pt] (m-23-22.south west) rectangle (m-22-23.north east);
\draw[fill=gray, opacity=0.5] (m-1-5.south west) rectangle (m-1-7.north east);
\draw[fill=gray, opacity=0.5] (m-2-6.south west) rectangle (m-2-7.north east);
\draw[fill=gray, opacity=0.5] (m-3-9.south west) rectangle (m-3-10.north east);
\draw[fill=gray, opacity=0.5] (m-5-12.south west) rectangle (m-4-12.north east);
\draw[fill=gray, opacity=0.5] (m-9-12.south west) rectangle (m-8-12.north east);
\draw[fill=gray, opacity=0.5] (m-11-14.south west) rectangle (m-11-14.north east);
\draw[fill=gray, opacity=0.5] (m-12-15.south west) rectangle (m-8-15.north east);
\draw[fill=gray, opacity=0.5] (m-13-17.south west) rectangle (m-13-18.north east);
\draw[fill=gray, opacity=0.5] (m-15-19.south west) rectangle (m-4-19.north east);
\draw[fill=gray, opacity=0.5] (m-14-23.south west) rectangle (m-13-23.north east);
\draw[fill=gray, opacity=0.5] (m-16-20.south west) rectangle (m-16-20.north east);
\draw[fill=gray, opacity=0.5] (m-17-23.south west) rectangle (m-16-23.north east);
\draw[fill=gray, opacity=0.5] (m-20-21.south west) rectangle (m-20-21.north east);
\draw[fill=gray, opacity=0.5] (m-21-23.south west) rectangle (m-19-23.north east);
\end{tikzpicture}\\

\caption {The canonical component is defined as in \cite[2.8]{FJ2}.
It is given as $\overline{B\cdot \mathfrak u}$, where $\mathfrak u$ is the root space complement in $\mathfrak m$ to the shaded areas of the matrix in the figure above. There is no $1$ in the shaded area whereas every $\ast$ does appear \cite [Cor. 6.9.3]{FJ2}.  The first property is also shared by the additional vectors corresponding to the $VS$ pairs \cite [Cor. 4.5.2]{FJ3}. One has $\overline{P\cdot E_{VS}}=\overline{B\cdot \mathfrak{u}}$ \cite [Prop. 4.5.3]{FJ3}.
One checks that for this equality only the bad $VS$ pair (noted in Figure $2$) is needed.}


\end {center}
\end {figure}


%
\end{document}